\documentclass{article}
\usepackage[margin=3cm]{geometry}
\usepackage{authblk}

\usepackage{amsthm, amssymb, latexsym, graphics}
\usepackage{amsmath,amsfonts}
\usepackage{hyperref}
\usepackage{todonotes}

\usepackage[giveninits = true, doi = false, url = false, isbn = false]{biblatex}
\newtheorem{theorem}{Theorem}[section]
\newtheorem{lemma}[theorem]{Lemma}
\newtheorem{corollary}[theorem]{Corollary}
\newtheorem{remark}[theorem]{Remark}

\DeclareMathOperator*{\esssup}{ess\,sup}

\newcommand{\D}{\mathbb{D}}

\newcommand\R{{\mathbb R}}

\newcommand{\linomega}{\zeta}

\title{Large-time behavior of the 2D thermally non-diffusive Boussinesq equations with Navier-slip boundary conditions}
\author[1]{Fabian Bleitner}
\affil[1]{McMaster University}
\author[2]{Elizabeth Carlson}
\affil[2]{California Institute of Technology}
\author[3]{Camilla Nobili}
\affil[3]{University of Surrey}
\date{}

\begin{document}

\maketitle

\begin{abstract}
This paper investigates the large-time behavior of a buoyancy-driven fluid without thermal diffusion under Navier-slip boundary conditions in a bounded domain with Lipschitz-continuous second derivatives. After establishing improved regularity for classical solutions, we analyze their large-time asymptotics. Specifically, we show that the solutions converge to a state where, as $t \rightarrow \infty$, $\|u\|_{W^{1,p}} \rightarrow 0$, and hydrostatic balance is achieved in the weak topology of $L^2$. Furthermore, we identify the necessary conditions under which stable stratification and hydrostatic balance can be achieved in the strong topology as time approaches infinity. We then analyze a particular steady state, the hydrostatic equilibrium, characterized by $ u = 0 $, $ \theta = \beta x_2 + \gamma $, and $ p = \frac{\beta}{2}x_2^2 + \gamma x_2 + \delta $. In a periodic strip, we establish the linear stability of this state for $\beta > 0$, indicating that the temperature is vertically stably stratified. This work builds upon the results in \cite{doeringWuZhaoZhen2018}, which focus on free-slip boundary conditions, as well as recent studies \cite{Aydin_Kukavica_Ziane_2023, aydin2024fractional} that address no-slip boundary conditions. Notably, the novelty of this study lies in the ability to directly bound the pressure term, made possible by the Navier-slip boundary conditions.
\end{abstract}

\section{Introduction}
In real-world phenomena, fluid behavior is influenced by external forces, boundary conditions, and various physical factors, including active scalars that characterize the flow. In the ocean, temperature and salinity are common examples of these scalars. This paper focuses on the Boussinesq approximation, where density variations due to temperature differences are retained only for generating buoyancy forces under the influence of gravity. Here, the active scalar (typically representing temperature or density), which diffuses and is advected by the fluid, affects the flow solely through buoyancy forces.
The Boussinesq system is given by
\begin{align*}
    u_t + u\cdot \nabla u - \nu \Delta u 
    &= 
    - \nabla p + \theta \hat{k}
    \\
    \nabla \cdot u &= 0
    \\
    \theta_t + u\cdot \nabla \theta 
    &=
    \mu \triangle \theta
    \\
    (u,\theta)(x,0)&=(u_0,\theta_0)(x)
\end{align*}
where $u$ is the velocity, $\theta$ is the active scalar, $\hat{k}$ represents the vertical direction ($\hat{k} = e_2$ in 2D, $\hat{k} = e_3$ in 3D), $\nu$ is the viscosity, and $\mu$ is the diffusivity of the active scalar.
In full space with $\nu, \mu >0$, the system has a unique global weak solution with initial data in $L^p$, with improved regularity of the solution when the initial data is smooth \cite{Cannon_DiBenedetto_1980}.
Limiting cases are natural extensions for study, allowing one to gain insight into dominating behaviors of the system.  For example, compressible ($\nabla \cdot u \neq 0$) adiabatic ($\mu = 0$) and inviscid ($\nu = 0$) flows are commonly considered in astrophysics (see, e.g., \cite{WoodBushby_2016}).

The particular limiting case we consider in this paper is the case where the fluid has an active scalar whose diffusivity is negligible, i.e. $\mu = 0$.  
In particular we consider the 2D incompressible, thermally non-diffusive Boussinesq equations
 \begin{align}
    u_t + u\cdot \nabla u - \nu \Delta u 
&=
    -\nabla p + \theta e_2 \label{navierStokes}
\\
    \nabla \cdot u &= 0 \label{incompressibility}
\\
    \theta_t +u\cdot \nabla \theta &= 0 \label{advection}
\\
    (u,\theta)(x,0)&=(u_0,\theta_0)(x) \label{initialData}
\end{align}
One physical understanding of the choice of $\mu = 0$ is the assumption that the system is adiabatic, that is, no heat (or scalar) is exchanged with the surroundings through diffusion, and transport occurs purely through fluid motion.
The motivation for studying this problem is closely related to boundary layer theory \cite{Li_Xu_Zhu_2016}. Additionally, in two dimensions, the structure of the fully inviscid and thermally non-diffusive system -- while at most locally well-posed (see \cite{Chae_Imanuvilov_1999, Chae_Kim_Nam_1999} and references therein) -- bears similarities to the structure of the three-dimensional Euler equations for axisymmetric swirling flows. This similarity leads to analogous mathematical challenges in obtaining estimates. For further discussion on this connection, see \cite{doeringWuZhaoZhen2018}, and for recent progress, refer to \cite{Li_Wang_2023} and references therein.
Several studies have addressed the system's well-posedness in full space, the torus \cite{Chae_Kim_Nam_1999, Hou_Li_2005, Chae_2006, Hmidi_Keraani_2007, Danchin_Paicu_2008, Larios_Lunasin_Titi_2013, Hu_Kukavica_Ziane_2015, Kukavica_Wang_Ziane_2016}, and bounded domains with smooth boundaries, considering both no-slip \cite{Lai_Pan_Zhao_2011, Hu_Kukavica_Ziane_2013} and free-slip conditions \cite{doeringWuZhaoZhen2018, Jang_Kim_2023}.

In this paper, we are interested in the Navier-slip boundary conditions in the tangential direction and no-penetration conditions in the normal direction, that is
 \begin{align}
    (\D u\ n + \alpha u)\cdot \tau &=0, \label{navSlip}
    \\
    u\cdot n &= 0\label{nonPenetration}
\end{align}
where $n$ is the unit outward normal vector, $\tau = n^\perp = (-n_2,n_1)$ is the corresponding tangent vector, and $\D u=\frac{1}{2}(\nabla u + \nabla u^T)$ is the symmetric gradient.
The Navier-slip boundary conditions serve as an ``interpolation" between free-slip and no-slip conditions. They are more physically realistic than free-slip conditions and, in many cases, provide a more accurate representation of fluid behavior compared to no-slip conditions (see, e.g., \cite{Lauga_Brenner_Stone_2007_no_slip_bc}).

Kelliher \cite{Kelliher_2006} extended results on the existence, uniqueness, and regularity of solutions for the 2D incompressible Navier-Stokes equations with Navier-slip boundary conditions, building on earlier work \cite{Clopeau_Mikelic_Robert_1998, LopesFilho_NussenzveigLopes_Planas_2005}.
The global existence and uniqueness of the partially dissipative Boussinesq system \eqref{navierStokes}-\eqref{initialData} on general bounded domains with Navier-slip boundary conditions for non-smooth initial data is proven in \cite{Hu_et_al_2018_GWP_partialDissipBoussinesq}.
Using techniques from \cite{bleitnerNobili} we follow the main strategy of \cite{doeringWuZhaoZhen2018}, who study the same system with free-slip boundary conditions, and prove higher regularity of the solutions under the same conditions as in \cite{Hu_et_al_2018_GWP_partialDissipBoussinesq}.
Our first main result concerns the regularity of the system.
\begin{theorem}[Regularity]
\label{theorem_well_posedness}
\leavevmode
\begin{itemize}
    \item (Uniform Regularity)
        Let $\Omega$ be a $C^{2,1}$-domain, $0<\alpha\in W^{1,\infty}(\partial\Omega)$, $u_0\in H^2(\Omega)$ satisfy the corresponding incompressibility and boundary conditions and $\theta_0\in L^{\tilde r}(\Omega)$ with $\tilde r\geq 4$. Then solutions of \eqref{navierStokes}-\eqref{navSlip} satisfy
        \begin{align}
            u &\in L^\infty\left((0,\infty);H^2(\Omega)\right) \cap L^p\left((0,\infty);W^{1,p}(\Omega)\right) \label{reg-u}
            \\
            u_t &\in L^\infty\left((0,\infty);L^2(\Omega)\right)\cap L^2\left((0,\infty);H^1(\Omega)\right) \label{reg-ut}
            \\
            p &\in L^\infty\left((0,\infty);H^1(\Omega)\right) \label{reg-p}
            \\
            \theta &\in L^\infty\big((0,\infty);L^r(\Omega)\big) 
        \end{align}
        for any $2\leq p <\infty$ and $1\leq r \leq \tilde r$.
    \item (Higher regularity)
        If additionally $\Omega$ is a simply connected $C^{3,1}$-domain, $0<\alpha\in W^{2,\infty}(\partial\Omega)$ and the initial data satisfies $\theta_0\in W^{1,\tilde q}$ with $\tilde q \geq 2$, then
        \begin{align*}
            u &\in L^2\left((0,T);H^3(\Omega)\right) \cap L^{\frac{2p}{p-2}}\left((0,T);W^{2,p}(\Omega)\right)
            \\
            \theta &\in L^\infty\left((0,T);W^{1,q}(\Omega)\right)
        \end{align*}
        for any $T>0$, $2\leq p <\infty$ and $1\leq q\leq \tilde q$.
\end{itemize}
\end{theorem}
In \cite{Hu_et_al_2018_GWP_partialDissipBoussinesq} the authors establish global well-posedness results for strong solutions in the class $u\in L^{\infty}\left((0,T); H^1(\Omega)\right)\cap L^2\left((0,T); H^2(\Omega)\right)$ and $\theta\in L^{\infty}\left((0,T); L^{\infty}(\Omega)\right)$ for any $T>0$. 
Using techniques from \cite{bleitnerNobili}, Theorem \ref{theorem_well_posedness} extends these results by showing that the crucial a-priori bounds hold universally in time in a space of higher regularity. Specifically we obtain $u \in L^\infty\left((0,\infty);H^2(\Omega)\right) \cap L^p\left((0,\infty);W^{1,p}(\Omega)\right)$ for any $p<\infty$. Additionally, under more restrictive conditions, we prove $ u \in L^2\left((0,T);H^3(\Omega)\right) \cap L^{\frac{2p}{p-2}}\left((0,T);W^{2,p}(\Omega)\right)$ for any $T>0$ and $2\leq p<\infty$.

\vspace{0.5cm}

Our next goal is to show large-time asymptotic behavior of the solution to \eqref{navierStokes}--\eqref{navSlip}.
In the last ten years, there have been numerous works studying the large-time asymptotic behavior of the Boussinesq system (with full and different combinations of partial dissipation or diffusion) in either 2 or 3 dimensions: 
For the case $\mu=0$ and $\nu>0$ several studies have addressed different configurations. In two-dimensional domains, results include cases with $C^1$ boundaries \cite{Lai_Pan_Zhao_2011}, periodic domains \cite{Biswas_Foias_Larios_2013, Tao_et_al_2020}, and $C^{2,\gamma}$
polygonal boundaries with free-slip conditions \cite{doeringWuZhaoZhen2018, Li_Wang_Zhao_2020}. Studies on systems with no-slip boundary conditions can be found in \cite{Li_Wang_Zhao_2020, Aydin_Kukavica_Ziane_2023}, while results for periodic strips are presented in \cite{Castro_Cordoba_Lear_2019, Chen_Li_2022}. Additional boundary conditions and small initial data have been explored in \cite{Dong_Sun_2023}, and results with lower bounds for the 2D torus, full space, and periodic strips are discussed in \cite{Kiselev_Park_Yao_2022}.
In the case where both $\mu>0$ and $\nu>0$, studies on the three-dimensional full space can be found in \cite{Brandolese_Schonbeck_2012, Weng_2016}.
For mixed viscosity and diffusivity, various results are available for two-dimensional systems \cite{Lai_Wu_Xu_et_al_2021, Wan_Chen_Chen_2022, Zhong_2022, Kang_Lee_Nguyen_2024}, while three-dimensional results can be found in \cite{Ma_Li_2023}.

In \cite{doeringWuZhaoZhen2018}, the authors considered free-slip boundary conditions, and analyze the large-time behavior of the perturbation near a particular steady state called the hydrostatic equilibrium 
and showed that the $L^2$ norm of the velocity field and its first order spatial and temporal derivatives converge to zero as $t\rightarrow \infty$. As a consequence, their result demonstrates that the pressure and concentration stratify in the vertical direction in the weak topology.
In \cite{Kukavica_Massatt_Ziane_2023} the authors assumed no-slip boundary conditions for $u$ on a bounded domain and showed that for data satisfying $(u_0, \theta_0)\in(H^2\cap H_0^1)\times H^1$ with $\nabla \cdot u_0=0$, one has 
$$\|\nabla u(t)\|_{L^2}\underset{t\rightarrow \infty}{\longrightarrow}0, \qquad \|Au(t)-\mathbb{P}(\theta(t)e_2)\|_{L^2}\underset{t\rightarrow \infty}{\longrightarrow}0 \quad \mbox{ and } \|Au(t)\|_{L^2}\leq C\,, $$
where $\mathbb{P}=Id+\nabla (-\Delta)^{-1}{\rm{div}}$ is the Leray projection and $A=-\mathbb{P}\Delta$ is the Stokes operator.  In the follow-up paper \cite{Aydin_Kukavica_Ziane_2023}, the authors proved that, as $t\rightarrow \infty$, $\mathbb{P}(\theta(t)e_2)$ weakly converges to $0$ in the class of $L^2$ vector fields with $\nabla\cdot u=0$ and $u\cdot n=0$ on $\partial\Omega$. Most recently, in the same setting the authors of \cite{aydin2024fractional} extend the results in \cite{Kukavica_Massatt_Ziane_2023} to initial data with different regularity and prove that
the concentration $\theta(t)$ achieves a steady state $\bar{\theta}$ if and only if $\lim_{t\rightarrow \infty}(Id-\mathbb{P})(\theta(t)e_2)=\bar{\theta}e_2$, with $\|\bar{\theta}\|_{L^2}=\|\theta_0\|_{L^2}$
Moreover, if such convergence to a limiting steady state holds, it is proved that $\|A u(t)\|_{L^2}\underset{t\rightarrow \infty}{\longrightarrow}0$ and  $\|\nabla p-\theta(t)e_2\|_{L^2}\underset{t\rightarrow \infty}{\longrightarrow}0\,$.

Inspired by the results in \cite{doeringWuZhaoZhen2018}, \cite{Aydin_Kukavica_Ziane_2023} and \cite{aydin2024fractional} and using techniques developed in \cite{bleitnerNobili} we prove the following

\begin{theorem}[Large-Time Behavior]\label{new-th-large_time}
Let $\Omega$ be a $C^{2,1}$-domain, $0<\alpha\in W^{1,\infty}(\partial\Omega)$, $u_0\in H^2(\Omega)$ satisfy the corresponding incompressibility and boundary conditions and $\theta_0\in L^{4}(\Omega)$. 
\begin{enumerate}
    \item \label{largetimetheorem_part_decay}
    As $t\to \infty$ we have 
    \begin{enumerate}
    \item  \begin{equation*}\label{largetimetheorem_decay_u}
     u(t)\to 0 \quad \text{ in }W^{1,q} \mbox{ for any } 1\leq q <\infty
     \end{equation*}
    \item
    \begin{equation*}\label{largetimetheorem_decay_ut}
    u_t(t)\to 0 \quad \text{ in }L^2
    \end{equation*}
    \item
    \begin{equation*} 
   \Delta u\rightharpoonup 0, \qquad\nabla p(t) - \theta(t) e_2\rightharpoonup 0 \quad \text{ in }L^2
    \label{largetimetheorem_decay_pTheta_weak}
    \end{equation*}
\end{enumerate}
    \item\label{largetimetheorem_part_decomposed}
     Let $\xi\in L^2$ with $\nabla\cdot \xi=0$ and $\xi\cdot n=0$ on $\partial\Omega$ and $\nabla \chi \in L^2$ with $\int_{\Omega} \chi=0$ be the divergence-free and curl-free part of the vector $\theta e_2$ respectively, i.e. $\theta e_2=\xi+\nabla \chi$. Then, as $t\rightarrow \infty$
    \begin{enumerate}
    \item
    \begin{equation*}
    \nabla p(t) - \nabla \chi(t) \to 0 \quad \text{ in }L^2
    \label{largetimetheorem_decay_p_grad_strong}
    \end{equation*}
    \item
    \begin{equation*}
    \xi(t) \rightharpoonup 0 \quad  \text{ in }L^2
    \label{largetimetheorem_decay_Projtheta_weak}
    \end{equation*}
    \item
    \begin{equation*}\label{largetimetheorem_decay_DeltaUprojTheta}
    \nu \Delta u(t) + \xi(t) \to 0 \quad  \text{ in }L^2
    \end{equation*}
\end{enumerate}
\item \label{largetimetheorem_part_assume}
Suppose that $\theta e_2$ converges to a steady state $\bar{\theta}$ in $L^2$. Then as $t\rightarrow \infty$
\begin{equation*}\label{full-conv-H2}
 \Delta u(t)\to 0 \quad \text{ in } L^2
\end{equation*}
and 
\begin{equation}\label{main-conv-results}
 \nabla p(t)-\theta(t) e_2\to 0 \quad \text{ in } L^2
 \end{equation}
\end{enumerate}
\end{theorem}

Notice that under the assumptions in part \ref{largetimetheorem_part_assume} we have $\|\xi\|_{L^2}\rightarrow 0$ as $t\rightarrow \infty$ (see the \hyperlink{proof_convergence_theorem_part_assumption}{proof} of Theorem \ref{new-th-large_time}). This implies that
$$\bar{\theta}e_2=\lim_{t\rightarrow \infty}\nabla \chi\,\qquad \mbox{in } L^2$$
which means that $\theta e_2$ must be curl free in the limit, and, in particular $0=\nabla\times(\bar{\theta}e_2)=\partial_1 \bar{\theta}$. The fact that $\bar{\theta}$ is independent of $x_1$ means that the fluid is vertically stratified, while the convergence $\|\nabla p-\theta(t)e_2\|_{L^2}\underset{t\rightarrow \infty}{\longrightarrow}0$ implies the hydrostatic balance is achieved (in the limit) in the strong topology. 
Parts \ref{largetimetheorem_part_decomposed} and \ref{largetimetheorem_part_assume} of our theorem were inspired by recent results in \cite{aydin2024fractional} and \cite{Kukavica_Massatt_Ziane_2023}.

We note that, unlike in \cite{aydin2024fractional} and \cite{Kukavica_Massatt_Ziane_2023}, we do not project our equations onto the space of divergence-free vector fields. While such a projection would simplify the analysis, it is not compatible with the boundary conditions considered here. Instead, taking advantage of the improved regularity properties of the flow under Navier-slip boundary conditions, we work directly with the equations and estimate the pressure, following ideas from \cite{drivasNguyenNobili2022,bleitnerNobili}. Additionally, under the assumptions of part \ref{largetimetheorem_part_assume}, we are able to demonstrate full convergence of the Laplacian (not just its divergence-free component) as a consequence of \eqref{largetimetheorem_decay_DeltaUprojTheta}.

One particular steady state that is worth discussing in this context is the hydrostatic equilibrium considered in \cite{doeringWuZhaoZhen2018}: Suppose for a moment that $\nu\neq 0$ and $\mu\neq 0$ and that we are looking for a solution of the form $(0, p_{\rm{he}}(x_2),\theta_{\rm{he}}(x_2))$. When formally substituting this Ansatz in \eqref{navierStokes} and \eqref{advection} we obtain $\partial_{x_2}p_{\rm{he}}(x_2)=\theta_{\rm{he}}(x_2)$ and $\partial_{x_2}^2 \theta(x_2)=0$ respectively. From the last equation we deduce that $\theta_{\rm{he}}(x_2)$ must be an affine function of the depth of the form $\theta(x_2)=\beta x_2+\gamma$.
In the case of the fully dissipative system ($\mu\neq 0$ and $\nu\neq 0$), under free-slip or no-slip boundary conditions an easy argument shows that the steady solution $(0, \beta x_2+\gamma)$ is globally asymptotically stable when $\beta>0$ (see Section 1.3 in \cite{doeringWuZhaoZhen2018}). This motivates us to study the stability of this Ansatz in the particular circumstances when the thermal diffusion is insignificant and the velocity satisfies Navier-slip boundary conditions on an arbitrary but regular domain. The asymptotic stability around the hydrostatic equilibrium of the solution of the Boussinesq system without molecular diffusivity has been studied under free-slip boundary conditions on a rectangular domain \cite{doeringWuZhaoZhen2018}, on the periodic strip $\mathbb{T}\times (0,1)$ \cite{Dong_2021} and on an infinite strip $\mathbb{R}\times (0,1)$ \cite{Dong_Sun_2022}. 

We perturb system \eqref{navierStokes}-\eqref{navSlip}, defining 
\begin{equation}
    \tilde u = u - u_{\textnormal{he}}, \quad
    \tilde\theta = \theta - \theta_{\textnormal{he}},\quad
    \tilde p = p - p_{\textnormal{he}}
\end{equation}
where 
\begin{equation}\label{HEq}
    u_{\textnormal{he}} = 0, \quad
    \theta_{\textnormal{he}} = \beta x_2 + \gamma, \quad
    p_{\textnormal{he}} = \frac{\beta}{2} x_2^2 + \gamma x_2 + \delta.
\end{equation}
Then $(\tilde u, \tilde\theta,  \tilde p)$
satisfy
\begin{align} 
    \tilde u_t + \tilde u\cdot \nabla \tilde u - \nu \Delta \tilde u 
&=
    -\nabla \tilde p + \tilde \theta e_2
    \label{tilde_navierStokes}
\\
    \nabla \cdot \tilde u &= 0
\\
    \tilde \theta_t +\tilde u\cdot \nabla \tilde \theta + \beta \tilde u_2 &= 0
    \\
    (\tilde u,\tilde \theta)(x,0)&=(\tilde u_0,\tilde \theta_0)(x)
    \\
    \tilde u \cdot n &= 0
    \\
    (\D \tilde u\ n + \alpha \tilde u)\cdot \tau&= 0
    \label{tilde_navSlip}
\end{align}
where $\tilde \theta_0 = \theta_0-\beta x_2+\gamma$ and $\tilde u_0=u_0$. Note that this system is exactly \eqref{navierStokes}-\eqref{navSlip}, just rewritten in the $\,\widetilde \cdot$ variables, and therefore Theorems \ref{theorem_well_posedness} and \ref{new-th-large_time} apply to \eqref{tilde_navierStokes}-\eqref{tilde_navSlip}.
Linearizing we obtain the system
\begin{align}
    U_t - \nu \Delta U + \nabla P &= \Theta e_2\label{linearized_pde_u}
    \\
    \nabla \cdot U &= 0
    \label{linearized_incompressible}
    \\
    \Theta_t +\beta U_2 &= 0 \label{linearized_pde_theta}
    \\
    (U,\Theta)(x,0)&=(U_0,\Theta_0)(x)
    \\
    U \cdot n &= 0
    \label{linearized_no_penetration}
    \\
    (\D U\ n + \alpha U)\cdot \tau&= 0 \label{linearized_navier_slip}
\end{align}
where we use capital letters to distinguish from the variables of the nonlinear system. The corresponding vorticity will be denoted by $\linomega$. Note that the eigenvalue problem for spatially periodic solutions to \eqref{linearized_pde_u}-\eqref{linearized_pde_theta} coincides with Theorem 1.4 (3) in \cite{doeringWuZhaoZhen2018}, implying instability when $\beta<0$. If however $\beta>0$ we are able to prove the linear stability of \eqref{linearized_pde_u}-\eqref{linearized_navier_slip}. Indeed, this makes physical sense as denser parcels of fluid sink while lighter parcels rise.

\begin{theorem}[Linearized System]
\label{theorem_linearStability}
Let $\Omega=\mathbb{T}\times(0,1)$, $\beta>0$, $\alpha\in \R$ and assume $\theta_0\in H^2(\Omega)$ and $U_0\in H^2(\Omega)$ satisfy the corresponding incompressibility and boundary conditions.
\begin{itemize}
    \item (Regularity)
        Then solutions of \eqref{linearized_pde_u}-\eqref{linearized_navier_slip} satisfy
        \begin{align}
            U&\in L^\infty\left((0,\infty);H^2(\Omega)\right) \cap L^2\left((0,\infty);H^3(\Omega)\right)\label{stabilityTheorem_reg_U}
            \\
            U_t&\in L^\infty\left((0,\infty);H^1(\Omega)\right) \cap L^2\left((0,\infty);H^2(\Omega)\right)
            \\
            \Theta &\in L^\infty\left((0,\infty);H^2(\Omega)\right)
            \\
            P &\in L^\infty\left((0,\infty);H^1(\Omega)\right)
            \\
            P_t &\in L^\infty\left((0,\infty);H^1(\Omega)\right) \cap L^2\left((0,\infty);H^1(\Omega)\right).
        \end{align}
    \item (Linear Stability)
        Additionally the hydrostatic equilibrium is stable in the sense that
        \begin{align}
            \|U(t)\|_{H^2} &\to 0
            \label{stabilityTheorem_decay_U}
            \\
            \|\nabla P(t) - \Theta (t) e_2 \|_{L^2} &\to 0
            \label{stabilityTheorem_decay_ThetaP}
        \end{align}
        for $t\to\infty$.
\end{itemize}
\end{theorem}

The regularity result follows in a similar fashion to that of the nonlinear system. 

We observe that Theorem \ref{new-th-large_time} shows, as $t\rightarrow \infty$, the solution converges to a state that satisfies the hydrostatic balance in the \textit{weak topology} of $L^2$. However $\theta$ does not necessarily converge to the linear profile $\beta x_2+\gamma$ in $L^2$. From our analysis we only deduce
\begin{equation*}
  \|\theta(t) - \beta x_2 - \gamma\|_{L^2}\to C \quad \mbox{ as } t\to \infty
\end{equation*}
(see Remark \ref{about-non-conv-theta}). We briefly remark that in all of the works previously cited, the only temperature steady state under consideration is actually a family of linear profiles, with no clear specification of the constants.  
To the best of the author's knowledge, the only study that explicitly investigates the exact equilibrium is \cite{Jang_Kim_2023}, where this profile is implicitly defined within the context of free-slip boundary conditions.
  However, the result is only proved for small initial data, and there is no explicit connection made to the linear profile.

The decay results \eqref{stabilityTheorem_decay_U} and \eqref{stabilityTheorem_decay_ThetaP} show that the hydrostatic equilibrium \eqref{HEq} is linearly stable whenever $\beta$ is positive. That is, perturbations about the hydrostatic equilibrium \eqref{HEq} of the linearized system converge to the hydrostatic equilibrium in strong norms.

The absence of nonlinear terms allows us to work solely in Hilbert spaces. The main difficulty for this problem stems from the absence of conservation laws and maximum principles for \eqref{linearized_pde_theta}. The convergence in \eqref{stabilityTheorem_decay_U} is based on the uniform boundedness in time of  $\|\nabla \Theta\|_{L^2}$. We are able to show this when $\Omega=\mathbb{T}\times(0,1)$, taking advantage of the periodicity in the horizontal direction. We want to stress that the bound we derive for $\nabla \Theta$ is not straightforward: We compute 
\begin{align}
   \frac{d}{dt}\left( \|\nabla \linomega\|_{L^2}^2 + \|\linomega\|_{L^2}^2 + \frac{1}{\beta} \|\nabla^2 \Theta\|_{L^2}^2 + \frac{1}{\beta}\|\nabla \Theta\|_{L^2}^2  + \frac{2}{\beta} \int_{\partial\Omega} \alpha (\partial_2 \Theta)^2\right)
\end{align}
(see \eqref{linearized_combination_a}) and show that, due to cancellations, this is bounded by $ \|U\|_{H^1}^2+\|U_t\|_{H^1}^2$, which, in turn, converges to zero uniformly in time. As a byproduct, we obtain a even higher order regularity, that is, $U\in L^2((0,\infty); H^3(\Omega))$. We believe that the same ideas can be employed to show the same result in the infinite strip $\Omega=\mathbb{R}\times(0,1)$.

We remark that, differently from the results of Dong \cite{Dong_2021} and Dong \& Sun \cite{Dong_Sun_2022,Dong_Sun_2023}, with the methods employed in this paper, we are not able to give explicit decay rates for the velocity, temperature and their derivatives.

The paper is structured as follows: Section \ref{section_proof_of_well_posedness} presents a series of consecutive a priori bounds leading to the \hyperlink{proof_of_theorem_well_posedness}{proof} of Theorem \ref{theorem_well_posedness}. Section \ref{section_large_time_behavior} establishes the convergence to hydrostatic equilibrium, as detailed in Theorem \ref{new-th-large_time}. Section \ref{section_linear_stability} focuses on the linearized system, providing the \hyperlink{proof_theorem_linearStability}{proof} of Theorem \ref{theorem_linearStability}. The Appendix includes \hyperlink{subsection:GradientEst}{gradient estimates} relevant to our geometry and boundary conditions, as well as \hyperlink{subsection:technicalLemmas}{technical lemmas} used for the main results.

\section{Proof of Theorem \ref{theorem_well_posedness}}\label{section_proof_of_well_posedness}

Before starting our analysis, we recall that the signed curvature $\kappa$ of a planar curve is defined as the rate of change of its tangent vector with respect to the curves arc length, i.e.
\begin{align}
    \label{def_kappa}
    (\tau \cdot \nabla) \tau = \kappa n.
\end{align}
For notational simplicity, since $u \cdot n = 0$ on the boundary, throughout the paper we will denote $u_\tau := u\cdot \tau$.

Unless otherwise stated, in the following sections we will always assume  
$$\Omega \mbox{ is a } C^{1,1}\mbox{-domain} \mbox{ and } 0<\alpha\in L^\infty(\partial\Omega) \mbox{  almost everywhere on } \partial\Omega.$$
Note that $\Omega$ is a $C^{k,1}$-domain if locally its boundary $\partial\Omega$ can be described as a $C^{k,1}$ map, i.e. the map is $k$-th order differentiable and each derivative is Lipschitz-continuous. For details see Section 1.2.1 in \cite{grisvard1985}.

 Finally, we remark that, as a solution of the transport equation, $\theta$ fulfills
\begin{align}
    \label{theta_conservation}
    \|\theta\|_{L^p}=\|\theta_0\|_{L^p}
\end{align}
for all $1\leq p \leq \infty$ if $\theta_0\in L^p$.

\subsection{Argument for \texorpdfstring{$u \in L^{\infty}((0,\infty); W^{1,p}(\Omega))$}{W1p-regularity} for \texorpdfstring{$p\geq 2$}{p bigger or equal 2}}\label{Sec:u-W1p_regularity}
We start by proving $H^1$-regularity by standard energy estimates. 
\begin{lemma}[Energy Bound]
\label{lemma_energy_bound}
Let $u_0,\theta_0\in L^2(\Omega)$. Then the energy is bounded by
\begin{equation*}
    \|u\|_{L^2} \leq e^{- \frac{\nu}{C} t}\|u_0\|_{L^2} + C \nu^{-1} \|\theta_0\|_{L^2}
\end{equation*}
for some constant $C = C(\alpha,\Omega)>0$.
\end{lemma}
\begin{proof}
Testing \eqref{navierStokes} with $u$,  we find
\begin{align}
\label{testingNSEwithU}
    \frac{1}{2}\frac{d}{dt} \|u\|_{L^2}^2 = - \int_\Omega u \cdot (u\cdot \nabla) u + \nu \int_\Omega u\cdot \Delta u - \int_\Omega u\cdot \nabla p + \int_\Omega u_2 \theta.
\end{align}
notice that by \eqref{nonPenetration} the first and the third term on the right-hand side of \eqref{testingNSEwithU} vanish under partial integration, which together with \eqref{lemma_uv_identities_second_id} implies
\begin{align*}
    \frac{1}{2}\frac{d}{dt} \|u\|_{L^2}^2 + 2 \nu \|\D u\|_{L^2}^2 + 2 \nu \int_{\partial\Omega}\alpha u_\tau^2 = \int_\Omega u_2 \theta.    
\end{align*}
Now Lemma \ref{lemma_coercivity} yields, for a constant $C>0$ depending only on $\alpha$ and $\Omega$,
\begin{align*}
    \frac{1}{2}\frac{d}{dt} \|u\|_{L^2}^2 + \frac{\nu}{C}\|u\|_{H^1}^2 &\leq \int_\Omega u_2 \theta
    \\
    &\leq \epsilon\|u\|_{L^2}^2 + \frac{1}{4\epsilon} \|\theta\|_{L^2}^2
\end{align*}
for any $\epsilon >0$. 
Choosing $\epsilon< \frac{\nu}{2C}$, \eqref{theta_conservation} and Gr{\"o}nwall's inequality yields
\begin{align}
    \frac{d}{dt} \|u\|_{L^2}^2 + \frac{\nu}{C}\|u\|_{H^1}^2 &\leq C \nu^{-1} \|\theta_0\|_{L^2}^2 
    \label{energy_decay_estimate_ode_uL2}
    \\
    \|u\|_{L^2}^2 &\leq e^{-\frac{\nu}{C} t}\|u_0\|_{L^2}^2 + C^2 \nu^{-2} \|\theta_0\|_{L^2}^2. \notag
\end{align}
\end{proof}

\begin{lemma}[\texorpdfstring{$H^1$}{H1}-Bound]
\label{lemma_L2H1_bound}
Let $u_0, \theta_0\in L^2(\Omega)$. Then $u\in L^2\left((0,\infty);H^1(\Omega)\right)$.
\end{lemma}
\begin{proof}
Let us define the new variables $\hat\theta = \theta - x_2$ and $\hat p = p - \frac{1}{2}x_2^2$. Then $u$, $\hat p$ and $\hat \theta$ solve
\begin{alignat}{2}
    u_t + u\cdot \nabla u - \nu \Delta u +\nabla \hat p &= \hat \theta e_2 &\qquad \text{ in } &\Omega \label{hat_navStokes}
    \\
    \nabla \cdot u &= 0 &\qquad \text{ in } &\Omega \label{hat_incompressibility}
    \\
    \hat\theta_t +u\cdot \nabla \hat \theta &= - u_2 &\qquad \text{ in } &\Omega\label{hat_advectionEquation}
    \\
    u \cdot n&=0 &\qquad \text{ on } &\partial\Omega \label{hat_nonPenetration}
    \\
    (\D u\ n + \alpha u)\cdot \tau  &= 0 &\qquad \text{ on } &\partial\Omega \label{hat_navSlip}
\end{alignat}
and testing \eqref{hat_navStokes} with $u$ and \eqref{hat_advectionEquation} with $\hat\theta$ we find
\begin{equation*}
    \begin{aligned}
        \frac{1}{2} \frac{d}{dt} \left( \|u\|_{L^2}^2 + \|\hat\theta\|_{L^2}^2\right)
        &= - \int_\Omega u \cdot (u\cdot \nabla ) u + \nu \int_\Omega u\cdot \Delta u - \int_\Omega \nabla \hat p \cdot u - \int_\Omega \hat \theta u\cdot \nabla \hat \theta.
    \end{aligned}
\end{equation*}
Notice that the first, third and fourth term on the right-hand side vanish under partial integration by \eqref{incompressibility} and \eqref{nonPenetration}. For the remaining diffusion term, \eqref{lemma_uv_identities_second_id} implies
\begin{align}
    \label{energy_identity}
    \frac{d}{dt} \left( \|u\|_{L^2}^2 + \|\theta- x_2\|_{L^2}^2\right) + 4 \nu \left(\|\D u\|_{L^2}^2 + \int_{\partial\Omega} \alpha u_\tau^2 \right) &= 0.
\end{align}
Integrating \eqref{energy_identity} in time,
\begin{equation*}
    \begin{aligned}
        \|u(t)\|_{L^2}^2 + \|\theta(t)- x_2\|_{L^2}^2 + 4 \nu  &\int_0^t \left(\|\D u(s)\|_{L^2}^2 + \int_{\partial\Omega} \alpha u_\tau^2(s) \right) \ ds
        \\
        &\qquad\qquad\qquad\qquad= \|u_0\|_{L^2}^2 + \|\theta_0 -x_2\|_{L^2}^2
    \end{aligned}
\end{equation*}
for all $t>0$, which by Lemma \ref{lemma_coercivity}, i.e. $\|u\|_{H^1}^2\leq C (\|\D u\|_{L^2}^2 + \int_{\partial\Omega} \alpha u_\tau^2)$, implies
\begin{align}
    \label{energy_int_uH1}
    \frac{\nu}{C} \int_0^t \|u(s)\|_{H^1}^2 \ ds \leq \|u_0\|_{L^2}^2 + \|\theta_0-x_2\|_{L^2}^2
\end{align}
with $C,\nu>0$, where the right-hand side of \eqref{energy_int_uH1} is independent of $t$ and therefore 
\begin{align}
    \label{uInL2H1}
    u\in L^2\left((0,\infty);H^1(\Omega)\right).
\end{align}
\end{proof}

Next, recall that the vorticity  $\omega =\nabla^\perp \cdot u$, where $\nabla^\perp=(-\partial_2,\partial_1)$, satisfies
\begin{alignat}{2}
    \omega_t + u\cdot \nabla \omega -\nu\Delta \omega &= \partial_1\theta & \qquad\text{ in }&\Omega\label{vorticity_eq}\\
    \omega &= -2(\alpha+\kappa) u_\tau & \qquad\text{ on }&\partial\Omega\label{vorticity_bc},
\end{alignat}
The boundary condition notice that as $\tau=(-n_2,n_1)$
\begin{equation}
    \label{vorticity_bc_derivation_1}
    \begin{aligned}
    \omega 
    &= \omega \tau\cdot \tau
    = \omega (-\tau_1 n_2 + \tau_2 n_1) 
    = 2 (\D u \ n)\cdot \tau - 2 n \cdot (\tau \cdot \nabla) u.
    \end{aligned}
\end{equation}
As $u$ is tangential to the boundary by \eqref{nonPenetration}, the last term in \eqref{vorticity_bc_derivation_1} can be rewritten as
\begin{align}
    \label{vorticity_bc_derivation_2}
    n\cdot (\tau\cdot\nabla) u = n\cdot (\tau\cdot\nabla) (u_\tau \tau)
    = \kappa u_\tau,
\end{align}
where in the last equality we used $n\cdot \tau = 0$ and the definition of $\kappa$, i.e. \eqref{def_kappa}. Combining \eqref{vorticity_bc_derivation_1}, \eqref{vorticity_bc_derivation_2} and \eqref{navSlip} results in the boundary condition \eqref{vorticity_bc}.

Notice since $u$ satisfies \eqref{incompressibility}, we have the identity
\begin{align}
    \label{nablaOmega_minusDeltaUperp}
    \nabla^\perp \omega = (\partial_2^2 u_1 -\partial_1\partial_2 u_2, -\partial_1\partial_2 u_1 + \partial_1^2 u_2) = \Delta u.
\end{align}

\begin{lemma}[\texorpdfstring{$W^{1,p}$}{W1p}-Bound]
\label{lemma_vorticity_bound}
Let $p\geq 2$ and $u_0\in W^{1,p}(\Omega)$. 
Then $u\in L^\infty\left((0,\infty);W^{1,p}(\Omega)\right)$.
\end{lemma}
The following proof is a slight variation of \cite[Lemma 3.7]{bleitnerNobili} and we state it here for the convenience of the reader.
\begin{proof}
Let $p>2$. 
Fix an arbitrary $T>0$, set $\Lambda = 2\|(\alpha+\kappa)u_\tau\|_{L^\infty(0,T\times \partial\Omega)}$ and let $\tilde\omega^\pm$ solve
\begin{alignat*}{2}
    \tilde\omega^\pm_t + u\cdot \nabla \tilde\omega^\pm -\nu \Delta \tilde\omega^\pm &= \partial_1 \theta & \qquad\text{ in }&\Omega\\
    \tilde\omega^\pm_0 &=\pm |\omega_0| & \qquad\text{ in }&\Omega\\
    \tilde\omega^\pm &= \pm \Lambda & \qquad\text{ on }&\partial\Omega.
\end{alignat*}
Then the difference $\bar\omega^\pm = \omega - \tilde\omega^\pm$ solves
\begin{alignat*}{2}
    \bar\omega^\pm_t + u\cdot \nabla \bar\omega^\pm -\nu \Delta \bar\omega^\pm &= 0 & \qquad\text{ in }&\Omega\\
    \bar\omega^\pm_0 &=\omega_0 \mp |\omega_0| & \qquad\text{ in }&\Omega\\
    \bar\omega^\pm &= -2 (\alpha+\kappa)u_\tau \mp \Lambda & \qquad\text{ on }&\partial\Omega.    
\end{alignat*}
Since the initial and boundary values of $\bar \omega^\pm$ are sign definite, this implies $\bar \omega^+ \leq 0$ ($\bar \omega^- \geq 0$) by the maximum principle. By the definition of $\bar \omega^\pm$ we therefore get $\tilde \omega^-\leq \omega \leq \tilde\omega^+$ and
\begin{align}
    \label{omega_leq_omegaTilde}
    |\omega|\leq \max\left\lbrace|\tilde\omega^+|,|\tilde\omega^-|\right\rbrace.
\end{align}
We will now derive upper bounds for $\tilde\omega^+$, which by symmetry also hold for $\tilde\omega^-$. Omitting the indices, we define
\begin{align}
    \label{def_omega_hat}
    \hat \omega = \tilde \omega - \Lambda,
\end{align}
which satisfies
\begin{alignat}{2}
    \hat\omega_t + u\cdot \nabla \hat \omega - \nu \Delta \hat \omega &= \partial_1 \theta & \qquad\text{ in }&\Omega\label{hatOmega_pde}\\
    \hat\omega_0 &= |\omega_0|-\Lambda & \qquad\text{ in }&\Omega\label{hotOmega_ic}\\
    \hat \omega &= 0 & \qquad\text{ on }&\partial\Omega.
\end{alignat}
Recall $p>2$ by assumption. Testing \eqref{hatOmega_pde} with $\hat \omega|\hat\omega|^{p-2}$ and integrating by parts we obtain
\begin{equation}
    \label{hatOmega_estimate_1}
    \begin{aligned}
        \frac{1}{p}\frac{d}{dt}\|\hat\omega\|_{L^p}^p 
        &= -(p-1) \nu \int_\Omega |\hat\omega|^{p-2} |\nabla\hat \omega|^2 - (p-1) \int_\Omega |\hat \omega|^{p-2} \theta \partial_1 \hat \omega.
    \end{aligned}
\end{equation}
Next we estimate the last term on the right-hand side of \eqref{hatOmega_estimate_1} via Hölder's and Young's inequality by
\begin{equation}
    \label{hatOmega_estimate_2}
    \begin{aligned}
        (p-1)\left\||\hat \omega|^{p-2} \theta \partial_1 \hat \omega\right\|_{L^1} &\leq (p-1) \left\|\theta\right\|_{L^p} \left\|\hat \omega^\frac{p-2}{2}\right\|_{L^q} \left\|\hat\omega^\frac{p-2}{2}\nabla \hat \omega\right\|_{L^2}
        \\
        &\leq \frac{\nu (p-1)}{2} \left\|\hat\omega^\frac{p-2}{2}\nabla \hat\omega \right\|_{L^2}^2 + \frac{p-1}{2\nu}\|\theta\|_{L^p}^2 \left\|\hat\omega^\frac{p-2}{2}\right\|_{L^q}^2,        
    \end{aligned}
\end{equation}
where $\frac{1}{p}+\frac{1}{q}+\frac{1}{2}=1$ and therefore $q=\frac{2p}{p-2}$. Combining \eqref{hatOmega_estimate_1} and \eqref{hatOmega_estimate_2} we find
\begin{align*}
    \frac{1}{p}\frac{d}{dt} \|\hat\omega\|_{L^p}^p + \frac{2(p-1)\nu}{p^2}\left\|\nabla \left(\hat\omega^\frac{p}{2}\right)\right\|_{L^2}^2 \leq \frac{p-1}{2\nu}\|\theta\|_{L^p}^2 \left\|\hat\omega\right\|_{L^p}^{p-2}.
\end{align*}
Next by Poincaré's inequality, since $\hat\omega = 0$ on $\partial\Omega$,
\begin{align*}
    \frac{1}{p}\frac{d}{dt} \|\hat\omega\|_{L^p}^p + \frac{\nu}{C}\left\|\hat\omega\right\|_{L^p}^p \leq \frac{p-1}{2\nu}\|\theta\|_{L^p}^2 \left\|\hat\omega\right\|_{L^p}^{p-2}
\end{align*}
for some constant $C>0$. We divide by $\|\hat\omega\|_{L^p}^{p-2}$, then use \eqref{theta_conservation} and Gr{\"o}nwall's inequality to obtain
\begin{align}
    \label{bound_omega_hat}
    \|\hat\omega\|_{L^p} \leq e^{-\frac{\nu}{C}t}\|\hat\omega_0\|_{L^p} + C\nu^{-1}\|\theta_0\|_{L^p}.
\end{align}
Next we estimate $\Lambda$. By Gagliardo-Nirenberg interpolation and Young's inequality we get
\begin{align}
\notag
        \Lambda &= 2 \|(\alpha+\kappa) u_\tau\|_{L^\infty((0,T)\times \partial\Omega)}
        \leq 2 \|\alpha+\kappa\|_{L^\infty(\partial\Omega)} \|u\|_{L^\infty((0,T)\times \Omega)}
        \\ \notag
        &\leq C \|\nabla u\|_{L^\infty(0,T;L^p(\Omega))}^\frac{p}{2(p-1)} \|u\|_{L^\infty(0,T;L^2(\Omega))}^{\frac{p-2}{2(p-1)}} + C \|u\|_{L^\infty(0,T;L^2(\Omega))}
        \\ \label{Lambda_estimate}
        &\leq \epsilon C \|\nabla u\|_{L^\infty(0,T;L^p(\Omega))} + C \left(1+\epsilon^{\frac{p}{2-p}}\right)\|u\|_{L^\infty(0,T;L^2(\Omega))}
\end{align}
for arbitrary $\epsilon>0$ and $C>0$ depends on $\alpha$, $\Omega$ and $p$.
Plugging the estimate of Lemma \ref{preliminaries_lemma_elliptic_regularity} for $k=0$, proven in the appendix, into \eqref{Lambda_estimate} and using Lemma \ref{lemma_energy_bound} we find
\begin{equation}
    \label{bound_lambda}
    \begin{aligned}
        \Lambda
        &\leq \epsilon C \|\omega\|_{L^\infty(0,T;L^p(\Omega))} + C \left(1 + \epsilon +  \epsilon^{\frac{2}{2-p}}\right)\left(\|u_0\|_{L^2} + \nu^{-1}\|\theta_0\|_{L^2}\right).
    \end{aligned}
\end{equation}
Recall that $|\omega|\leq \max\lbrace|\tilde\omega^+|,|\tilde\omega^-|\rbrace$ by \eqref{omega_leq_omegaTilde}. Combining \eqref{def_omega_hat}, \eqref{bound_omega_hat}, \eqref{hotOmega_ic} and \eqref{bound_lambda} we obtain
\begin{equation*}
    \begin{aligned}
        \|\omega\|_{L^\infty(0,T;L^p(\Omega))}
        &\leq \|\tilde \omega\|_{L^\infty(0,T;L^p(\Omega))}
        = \|\hat \omega + \Lambda \|_{L^\infty(0,T;L^p(\Omega))}
        \\
        &\leq \|\hat\omega_0\|_{L^p} + C \nu^{-1}\|\theta_0\|_{L^p} + C \Lambda
        \\
        &\leq \|\omega_0\|_{L^p} + C \nu^{-1}\|\theta_0\|_{L^p} + \epsilon C \|\omega\|_{L^\infty(0,T;L^p(\Omega))}
        \\
        &\qquad + C \left(1 + \epsilon + \epsilon^{\frac{2}{2-p}}\right)\left(\|u_0\|_{L^2} + \nu^{-1}\|\theta_0\|_{L^2}\right).
    \end{aligned}
\end{equation*}
Finally choosing $\epsilon< \frac{1}{C}$ and using Hölder's inequality we find
\begin{align}
    \label{lp_omega_temp}
    \|\omega\|_{L^\infty(0,T;L^p(\Omega))} \leq C\left(\|\omega_0\|_{L^p} + \|u_0\|_{L^2} +  \nu^{-1}\|\theta_0\|_{L^p}\right).
\end{align}
By \eqref{Lambda_estimate} also $\|u\|_{L^\infty(0,T;L^p(\Omega))}$ is bounded by the right-hand side of \eqref{lp_omega_temp}, which together with Lemma \ref{preliminaries_lemma_elliptic_regularity} implies
\begin{align*}
    \|u\|_{L^\infty(0,T;W^{1,p}(\Omega))} \leq C\left(\|\nabla u_0\|_{L^p} + \|u_0\|_{L^2} + \nu^{-1}\|\theta_0\|_{L^p}\right).
\end{align*}
As the bound is independent of $T$ it holds uniformly in time.

For $p=2$, the statement follows from the embedding $L^p\subset L^2$ for any $p>2$.
\end{proof}

\subsection{Argument for \texorpdfstring{$p\in L^\infty((0,\infty);H^1(\Omega))$}{H1-pressure-regularity}}\label{Sec:p-H1_regularity}

The pressure satisfies
\begin{alignat}{2}
    \Delta p &= - (\nabla u)^T: \nabla u + \partial_2 \theta &\qquad \text{ in } &\Omega \label{pressure_pde}
    \\
    n\cdot \nabla p &= - \kappa u_\tau^2 + 2 \nu \tau \cdot \nabla \left((\alpha+\kappa) u_\tau \right) + n_2\theta &\qquad \text{ on } &\partial\Omega \label{pressure_bc}
\end{alignat}

The equation in the bulk is obtained by taking the divergence of \eqref{navierStokes} and using \eqref{incompressibility}. In order to derive the boundary conditions, we trace the projection of \eqref{navierStokes} in the normal direction on the boundary
\begin{align}
    \label{n_cdot_navStokes}
    n\cdot u_t + n \cdot (u\cdot \nabla) u - \nu n\cdot \Delta u + n\cdot \nabla p = n_2 \theta\,.
\end{align}
The first term on the left-hand side of \eqref{n_cdot_navStokes} vanishes by \eqref{nonPenetration}. For the nonlinear term we find
\begin{align*}
    n\cdot (u\cdot \nabla) u = u_\tau n\cdot (\tau \cdot \nabla) u = \kappa u_\tau^2
\end{align*}
thanks to \eqref{nonPenetration} and \eqref{vorticity_bc_derivation_2}. By \eqref{nablaOmega_minusDeltaUperp} and \eqref{vorticity_bc} the diffusion term fulfills
\begin{align}
    n \cdot \Delta u
    = n\cdot \nabla^\perp \omega
    = - 2 n\cdot \nabla^\perp \left((\alpha+\kappa)u_\tau\right)
    = 2 \tau \cdot \nabla \left((\alpha+\kappa)u_\tau\right).
    \label{nCdotDeltaU_on_boundary}
\end{align}

\begin{lemma}[Pressure Bound]
\label{lemma_pressure_bound}
Suppose $u_0\in W^{1,4}(\Omega)$ and $\theta_0\in L^4(\Omega)$.  Then $p\in L^\infty((0,\infty);H^1(\Omega))$. 
\end{lemma}
\begin{proof}
Testing \eqref{pressure_pde} with $p$ and integrating by parts we obtain
\begin{align}
    \label{test_pressure_pde}
    \int_\Omega  p \Delta p
    &= - \int_\Omega p (\nabla u)^T: \nabla u + \int_\Omega p\partial_2 \theta
    =  - \int_\Omega p (\nabla u)^T: \nabla u + \int_{\partial\Omega} p n_2 \theta - \int_\Omega \theta \partial_2 p\,.
\end{align}
Again integration by parts and \eqref{pressure_bc} imply
\begin{equation}
    \label{identity_nablaP}
    \begin{aligned}
        \|\nabla p\|_{L^2}^2
        &= \int_{\partial\Omega} n p \cdot \nabla p - \int_\Omega p \Delta p
        \\
        &= - \int_{\partial\Omega} p \kappa u_\tau^2 + \int_{\partial\Omega} 2  p \nu \tau \cdot \nabla \left((\alpha+\kappa) u_\tau \right) +  \int_{\partial\Omega} p n_2 \theta - \int_\Omega p \Delta p.
    \end{aligned}
\end{equation}
Plugging \eqref{test_pressure_pde} into \eqref{identity_nablaP} we find
\begin{equation}
    \label{nablaP_identity_2}
    \begin{aligned}
        \|\nabla p\|_{L^2}^2
        &= - \int_{\partial\Omega} p \kappa u_\tau^2 + 2 \int_{\partial\Omega} p \nu \tau \cdot \nabla \left((\alpha+\kappa) u_\tau \right) + \int_\Omega p (\nabla u)^T: \nabla u + \int_\Omega \theta \partial_2 p.
    \end{aligned}
\end{equation}
We estimate the right-hand side of \eqref{nablaP_identity_2} individually. The first term can be bounded by the trace theorem, Hölder's inequality, and $\epsilon$-Young's inequality by
\begin{equation}
    \label{p_estimate_estimate_1}
    \begin{aligned}
        \left| \int_{\partial\Omega} p \kappa u_\tau^2\right|
        &\leq C \|\kappa\|_{L^\infty} \| p u^2\|_{W^{1,1}}
        \leq C \|p\|_{H^1} \|u^2\|_{H^1}
        \\
        &
        \leq C \|p\|_{H^1} \|u\|_{L^4} \|u\|_{W^{1,4}}
        \leq \epsilon \|p\|_{H^1}^2 + C \epsilon^{-1} \|u\|_{W^{1,4}}^4,        
    \end{aligned}
\end{equation}
the second term by Lemma \ref{lemma_nablaPcdotU_onBoundary} and $\epsilon$-Young's inequality
\begin{align}
    \label{p_estimate_estimate_2}
    \left|\int_{\partial\Omega} 2 p \nu \tau \cdot \nabla \left((\alpha+\kappa) u_\tau\right) \right| &\leq C \nu\|(\alpha+\kappa)n\|_{W^{1,\infty}(\partial\Omega)}\|p\|_{H^1}\|u\|_{H^1}
    \\
    &\lesssim \epsilon \|p\|_{H^1}^2 + C \epsilon^{-1} \|u\|_{H^1}^2,
\end{align}
the third term by
\begin{align}
    \label{p_estimate_estimate_3}
    \left\| p |\nabla u|^2\right\|_{L^1} \leq \|p\|_{L^2} \|\nabla u\|_{L^4}^2 \leq \epsilon \|p\|_{H^1}^2 + \epsilon^{-1} \|u\|_{W^{1,4}}^4,
\end{align}
and the fourth term by
\begin{align}
    \label{p_estimate_estimate_4}
    \|\theta \partial_2 p \|_{L^1} \leq \|\theta \|_{L^2} \|p\|_{H^1} \leq \epsilon \|p\|_{H^1}^2 + \epsilon^{-1} \|\theta\|_{L^2}^2.
\end{align}
Combining \eqref{nablaP_identity_2}, \eqref{p_estimate_estimate_1}, \eqref{p_estimate_estimate_2}, \eqref{p_estimate_estimate_3} and \eqref{p_estimate_estimate_4}
\begin{align*}
    \|\nabla p\|_{L^2}^2 \leq 4 \epsilon \|p\|_{H^1}^2 + C\epsilon^{-1} \left(\|u\|_{W^{1,4}}^4 + \nu^{2} \|u\|_{H^1}^2 + \|\theta\|_{L^2}^2\right).
\end{align*}
As $p$ is only defined up to a constant we choose it such that $p$ is average free. Therefore Poincaré's inequality yields
\begin{align*}
    \|p\|_{H^1}^2 \leq C\|\nabla p\|_{L^2}^2 \leq 4 C\epsilon \|p\|_{H^1}^2 + C\epsilon^{-1} \left(\|u\|_{W^{1,4}}^4 + \nu^{2} \|u\|_{H^1}^2 + \|\theta\|_{L^2}^2\right).
\end{align*}
Choosing $\epsilon$ sufficiently small we find
\begin{align*}
    \|p\|_{H^1}^2 \leq C\left(\|u\|_{W^{1,4}}^4 + \nu^{2} \|u\|_{H^1}^2 + \|\theta\|_{L^2}^2\right) \leq C\left(\|u\|_{W^{1,4}}^4 + \nu^{2} \|u\|_{W^{1,4}}^2 + \|\theta\|_{L^4}^2\right),
\end{align*}
where in the last inequality we used Hölder's inequality and applying the previous bounds for $u$ and $\theta$, i.e. Lemma \ref{lemma_vorticity_bound} and \eqref{theta_conservation},
\begin{align*}
    \|p\|_{H^1} &\leq C\left(\|u_0\|_{W^{1,4}}^2 + \nu^{-2}\|\theta_0\|_{L^4}^2+ \nu \|u_0\|_{W^{1,4}} + \|\theta_0\|_{L^4}\right)
    \\
    &\leq C\left(\|u_0\|_{W^{1,4}}^2 + \nu^{-2}\|\theta_0\|_{L^4}^2+\nu^2\right).
\end{align*}
\end{proof}


\subsection{Argument for \texorpdfstring{$u_t \in L^\infty\left((0,\infty);L^2(\Omega)\right)\cap L^2\left((0,\infty);H^1(\Omega)\right)$}{ut-regularity}}\label{Sec:u_t-regularity}
\begin{lemma}[\texorpdfstring{$u_t$}{ut}-Bound]
\label{lemma_ut_bound}
Let $\theta_0\in L^4(\Omega)$, $u_0\in H^2(\Omega)$, $\nabla \cdot u_0=0$ in $\Omega$ and  $(\D u_0\ n + \alpha u_0)\cdot \tau=0$ on $\partial \Omega$. Then $u_t \in L^\infty\left((0,\infty);L^2(\Omega)\right)\cap L^2\left((0,\infty);H^1(\Omega)\right)$.
\end{lemma}

\begin{proof}

Differentiating \eqref{navierStokes}, \eqref{incompressibility}, \eqref{nonPenetration} and \eqref{navSlip} with respect to time yields
\begin{alignat}{2}
    u_{tt} + u_t\cdot \nabla u + u\cdot \nabla u_t - \nu \Delta u_t +\nabla p_t &= \theta_t e_2 &\qquad \text{ in } &\Omega \label{navStokes_t}
    \\
    \nabla \cdot u_t &= 0 &\qquad \text{ in } &\Omega \label{incompressibility_t}
    \\
    u_t \cdot n &=0 &\qquad \text{ on } &\partial\Omega \label{nonPenetration_t}
    \\
    (\D u_t \ n +\alpha u_t) \cdot \tau &= 0 &\qquad \text{ on } &\partial\Omega.
    \label{navSlip_t}
\end{alignat}
Testing \eqref{navStokes_t} with $u_t$ and using \eqref{incompressibility_t}, \eqref{nonPenetration_t}, \eqref{navSlip_t} and analogous estimates as in the proof of Lemma \ref{lemma_energy_bound} we find
\begin{align*}
    \frac{1}{2}\frac{d}{dt}\|u_t\|_{L^2}^2 + 2 \nu \|\D u_t\|_{L^2}^2 + 2 \nu \int_{\partial\Omega} \alpha (\tau\cdot u_t)^2 = - \int_\Omega u_t \cdot (u_t\cdot\nabla) u + \int_\Omega \theta_t u_t \cdot e_2.
\end{align*}
Substituting the thermal evolution equation \eqref{advection} and using integration by parts, \eqref{nonPenetration}, and \eqref{incompressibility},
\begin{equation}
    \label{ut_ode_identity}
    \begin{aligned}
        \frac{1}{2}\frac{d}{dt}\|u_t\|_{L^2}^2 + 2 \nu \|\D u_t\|_{L^2}^2 + 2 \nu \int_{\partial\Omega} \alpha (\tau\cdot u_t)^2
        &= - \int_\Omega u_t \cdot (u_t\cdot\nabla) u + \int_\Omega \theta u \cdot \nabla (u_t\cdot e_2).        
    \end{aligned}
\end{equation}
By Hölder's, Ladyzhenskaya's, and Young's inequalities
\begin{equation}
    \label{ut_estimate_1}
    \begin{aligned}
        \int_\Omega \left|u_t\cdot (u_t\cdot \nabla ) u\right|
        &\leq \|\nabla u\|_{L^2}\|u_t\|_{L^4}^2
        \leq C\|\nabla u\|_{L^2}\left(\|\nabla u_t\|_{L^2}^\frac{1}{2}\|u_t\|_{L^2}^\frac{1}{2}+\|u_t\|_{L^2}\right)^2
        \\
        &\leq C \|\nabla u\|_{L^2} \|\nabla u_t\|_{L^2} \|u_t\|_{L^2} + C \|\nabla u\|_{L^2} \|u_t\|_{L^2}^2
        \\
        &\leq \epsilon \|u_t\|_{H^1}^2 + C \epsilon^{-1} \|\nabla u\|_{L^2}^2 \|u_t\|_{L^2}^2
    \end{aligned}
\end{equation}
for all $\epsilon>0$ and similarly
\begin{align}
    \label{ut_estimate_2}
    \int_\Omega \left| \theta u\cdot \nabla (u_t\cdot e_2)\right| 
    &\leq \|\theta\|_{L^4} \|u\|_{L^4} \|\nabla u_t\|_{L^2}
    \leq \epsilon \|\nabla u_t\|_{L^2}^2 + \epsilon^{-1} \| \theta\|_{L^4}^2\|u\|_{H^1}^2.
\end{align}
Combining \eqref{ut_ode_identity}, \eqref{ut_estimate_1} and \eqref{ut_estimate_2}
\begin{equation*}
    \begin{aligned}
        \frac{1}{2}\frac{d}{dt}\|u_t\|_{L^2}^2 &+ 2 \nu \|\D u_t\|_{L^2}^2 + 2 \nu \int_{\partial\Omega} \alpha (\tau\cdot u_t)^2 
        \\
        &\leq 2\epsilon \|u_t\|_{H^1}^2 + C\epsilon^{-1}\|\nabla u\|_{L^2}^2 \|u_t\|_{L^2}^2 + \epsilon^{-1} \|\theta\|_{L^4}^2\|u\|_{H^1}^2.
    \end{aligned}
\end{equation*}
Using the coercivity estimate (i.e. Lemma \ref{lemma_coercivity} applied to $u_t$) \eqref{theta_conservation}, and choosing $\epsilon$ sufficiently small,
\begin{align}
    \label{ut_ode_estimate}
    \frac{d}{dt}\|u_t\|_{L^2}^2 &\leq \frac{d}{dt}\|u_t\|_{L^2}^2 + \frac{1}{\tilde C} \|u_t\|_{H^1}^2\leq C\left( \|u_t\|_{L^2}^2 + \|\theta_0\|_{L^4}^2\right)\|u\|_{H^1}^2
\end{align}
for constants $C,\tilde C >0$ depending on $\Omega$, $\nu$, $\|\alpha^{-1}\|_{L^\infty}^{-1}$, $\|\alpha^{-1}\kappa\|_{L^\infty}^{-1}$. Grönwall's inequality yields
\begin{align}
    \label{ut_gronwall_estimate}
    \|u_t(t)\|_{L^2}^2 \leq e^{C\int_0^t\|u(s)\|_{H^1}^2 ds}\|u_t(0)\|_{L^2}^2 + C\|\theta_0\|_{L^4}^2\int_0^t e^{C\int_s^t\|u(r)\|_{H^1}^2\ dr} \|u(s)\|_{H^1}^2 \ ds\,.
\end{align}
Testing smooth solutions of \eqref{navierStokes}-\eqref{navSlip} with $u_t$ one gets
\begin{align*}
    \|u_t\|_{L^2}^2 &\leq \|u_t\|_{L^2}\|u\cdot \nabla u\|_{L^2} + \nu \|u_t\|_{L^2}\|\Delta u\|_{L^2} + \|u_t\|_{L^2}\|\theta\|_{L^2}
    \\
    &\leq \|u_t\|_{L^2}\left((\|u\|_{H^1} + \nu) \|u\|_{H^2} + \|\theta\|_{L^2}\right),
\end{align*}
implying $\|u_t(0)\|_{L^2}\leq (\|u_0\|_{H^1} + \nu) \|u_0\|_{H^2} + \|\theta_0\|_{L^2}$. Therefore, using Lemma \ref{lemma_L2H1_bound}, the right-hand side of \eqref{ut_gronwall_estimate} is universally bounded in time, implying
\begin{align}
    \label{ut_l2_bound_in_proof}
    u_t\in L^\infty\left((0,\infty);L^2(\Omega)\right).
\end{align}
Integrating \eqref{ut_ode_estimate} in time,
\begin{align*}
    \int_0^\infty \|u_t(s)\|_{H^1}^2\ ds\leq C\left(\|\theta_0\|_{L^4}^2+\esssup_{0\leq s\leq\infty} \|u_t(s)\|_{L^2}^2 \right) \int_0^\infty \|u(s)\|_{H^1}^2 \ ds + C \|u_t(0)\|_{L^2}^2 <\infty
\end{align*}
by Lemma \ref{lemma_L2H1_bound} and \eqref{ut_l2_bound_in_proof}, implying
\begin{align}
    \label{ut_in_L2H1}
    u_t \in L^2((0,\infty);H^1(\Omega)).
\end{align}
\end{proof}

\vspace{1cm}

We are now ready to prove Theorem \ref{theorem_well_posedness}.

\begin{proof}[Proof of Theorem \ref{theorem_well_posedness}]
\hypertarget{proof_of_theorem_well_posedness}{\phantom{a}}
\leavevmode
\begin{itemize}
    \item Regularity:
        In Section \ref{Sec:u-W1p_regularity} we showed that for any $p\geq 2$
        \begin{equation}\label{u-inter}u\in L^{\infty}((0,\infty); W^{1,p}).\end{equation}
        The regularity \eqref{reg-p} and \eqref{reg-ut} is proved in Section \ref{Sec:p-H1_regularity} and \ref{Sec:u_t-regularity} respectively. The regularity $u\in L^{\infty}((0,\infty), H^2)$ now follows by subtraction. Indeed, by \eqref{nablaOmega_minusDeltaUperp} and \eqref{navierStokes}
        \begin{equation}
            \label{nablaOmega_estimate}
            \begin{aligned}
                \nu^2 \|\nabla \omega\|_{L^2}^2 &= \nu^2\|\Delta u\|_{L^2}^2
                = \|u_t + u\cdot \nabla u+\nabla p -\theta e_2 \|_{L^2}^2
                \\
                &\leq C \left(\|u_t\|_{L^2}^2 + \|u\|_{W^{1,4}}^2 + \|p\|_{H^1}^2 + \|\theta\|_{L^2}^2\right).
            \end{aligned}
        \end{equation}
        Since $\Omega$ is a $C^{2,1}$ domain and $\alpha\in W^{2,\infty}$, Lemma \ref{preliminaries_lemma_elliptic_regularity} implies
        \begin{align}
            \label{uH2_estimate}
            \|u\|_{H^2}^2
            &\leq \|\nabla u\|_{H^1}^2 + \|u\|_{L^2}^2
            \leq C \|\omega\|_{H^1} + C \|u\|_{L^2}^2
            \leq C \|\nabla \omega\|_{L^2}^2 + C \|u\|_{H^1}^2. 
        \end{align}
        Combining \eqref{nablaOmega_estimate} and \eqref{uH2_estimate} and using H{\"o}lder's inequality,
        \begin{equation}
            \label{uH2_estimate2}
            \begin{aligned}
                \|u\|_{H^2}^2
                &\leq C \|\nabla \omega\|_{L^2}^2 + C \|u\|_{H^1}^2
                \\
                &\leq C \left(\|u_t\|_{L^2}^2 + \|u\|_{W^{1,4}}^2 + \|p\|_{H^1}^2 + \|\theta\|_{L^2}^2\right).
            \end{aligned}
        \end{equation}
     Notice that by Lemmas \ref{lemma_pressure_bound}, \ref{lemma_vorticity_bound}, and \ref{lemma_ut_bound} and \eqref{theta_conservation}, the right-hand side of \eqref{uH2_estimate2} is universally bounded in time with a constant $C = C(\alpha,\Omega,\nu, \kappa) > 0$, implying
        \begin{align}
            \label{H2}
            u\in L^\infty((0,\infty);H^2(\Omega)).
        \end{align}
        Finally the regularity of $u$ in \eqref{reg-u} follows from \eqref{u-inter},  \eqref{H2}, and the Gagliardo-Nirenberg interpolation inequality, 
        \begin{align*}
            \|u\|_{W^{1,p}}^p\leq C \|u\|_{H^2}^{p-2} \|u\|_{H^1}^{2}\,,
        \end{align*}
        for any $2\leq p <\infty$. 

    \item Higher Regularity:
        Now let $\Omega$ be a simply connected domain with $C^{3,1}$ boundary and $0<\alpha \in W^{2,\infty}(\partial\Omega)$. Testing \eqref{vorticity_eq} with $\Delta\omega$ we find
        \begin{align}
            \label{vorticity_testDeltaOmega}
            \int_\Omega \omega_t \Delta \omega + \int_\Omega u\cdot \nabla \omega \Delta \omega -\nu \|\Delta\omega\|_{L^2}^2 =\int_\Omega \partial_1\theta \Delta\omega.
        \end{align}
        The first term on the left-hand side can be written as
        \begin{align*}
            \int_\Omega \omega_t \Delta\omega = - \int_\Omega \nabla \omega \cdot \nabla\omega_t + \int_{\partial\Omega} \omega_t n\cdot\nabla \omega = -\frac{1}{2}\frac{d}{dt}\|\nabla\omega\|_{L^2}^2 + \int_{\partial\Omega} \omega_t n\cdot\nabla \omega,
        \end{align*}
        which combined with \eqref{vorticity_testDeltaOmega} implies
        \begin{equation}
            \label{h3nablaOmega_estimate}
            \begin{aligned}
                \frac{1}{2}\frac{d}{dt}\|\nabla\omega\|_{L^2}^2 &+ \nu \|\Delta\omega\|_{L^2}^2 = \int_{\partial\Omega} \omega_t n\cdot\nabla \omega + \int_\Omega u\cdot\nabla\omega \Delta\omega - \int_\Omega \partial_1 \theta \Delta\omega.
            \end{aligned}
        \end{equation}
        We estimate these terms individually. Since $\alpha$, $\kappa$ and $\tau$ are independent of time, \eqref{vorticity_bc} implies
        \begin{align*}
            \omega_t = -2 (\alpha+\kappa) u_t\cdot \tau
        \end{align*}
        on $\partial\Omega$, which combined with Hölder's inequality, the trace theorem and Young's inequality yields
        \begin{equation}
            \label{h3nablaOmega_estimate_part1}
            \begin{aligned}
                \int_{\partial\Omega} \omega_t n\cdot\nabla \omega &= -2\int_{\partial\Omega} (\alpha+\kappa) u_t\cdot \tau n\cdot\nabla\omega
                \leq C\|\alpha+\kappa\|_{L^\infty} \|u_t \nabla\omega\|_{W^{1,1}}
                \\
                &\leq C \|u_t\|_{H^1} \|\nabla\omega\|_{H^1}
                \leq \epsilon \|\nabla\omega\|_{H^1}^2 + \epsilon^{-1} C \|u_t\|_{H^1}^2
            \end{aligned}
        \end{equation}
        for any $\epsilon>0$. For the second term on the right-hand side of \eqref{h3nablaOmega_estimate}, Hölder's inequality, Ladyzhenskaya's interpolation inequality, and Young's inequality imply
        \begin{equation}
            \label{h3nablaOmega_estimate_part2}
            \begin{aligned}
                \int_\Omega u\cdot \nabla \omega \Delta \omega
                &\leq \|u\|_{L^4} \|\nabla \omega\|_{L^4} \|\Delta\omega\|_{L^2}
                \leq \|u\|_{H^1}^\frac{1}{2}\|u\|_{L^2}^\frac{1}{2}\|\nabla\omega\|_{H^1}^\frac{1}{2}\|\nabla\omega\|_{L^2}^\frac{1}{2}\|\Delta\omega\|_{L^2}
                \\
                &\leq C\|u\|_{H^1} \|\nabla\omega\|_{H^1}^\frac{3}{2} \|\nabla\omega\|_{L^2}^\frac{1}{2}
                \leq \epsilon \|\nabla\omega\|_{H^1}^2 + C\epsilon^{-1} \|u\|_{H^1}^4 \|\nabla \omega\|_{L^2}^2.
            \end{aligned}
        \end{equation}
        Notice also that Lemma \ref{preliminaries_lemma_elliptic_regularity} applied to \eqref{h3nablaOmega_estimate_part1} and \eqref{h3nablaOmega_estimate_part2} yields
        \begin{align}
            \label{h3nablaOmega_estimate_a}
            \int_{\partial\Omega} \omega_t n\cdot\nabla \omega \leq \epsilon C \|\Delta\omega\|_{L^2}^2 + \epsilon^{-1}C\|u_t\|_{H^1}^2 + \epsilon C \|u\|_{H^2}^2
        \end{align}
        and
        \begin{align}
            \label{h3nablaOmega_estimate_b}
            \int_\Omega u\cdot \nabla \omega \Delta \omega \leq \epsilon C \|\Delta\omega\|_{L^2}^2 + C \left(\epsilon + \epsilon^{-1} \|u\|_{H^1}^4 \right) \|u\|_{H^2}^2.
        \end{align}
        Let at first $q\geq 2$, then the last term on the right-hand side of \eqref{h3nablaOmega_estimate} can be estimated by Young's inequality as
        \begin{align}
            \label{h3nablaOmega_estimate_part4}
            -\int_\Omega \partial_1 \theta \Delta \omega
            &\leq \|\nabla\theta\|_{L^2}\|\Delta\omega\|_{L^2} \notag
            \\
            &\leq \epsilon \|\Delta\omega\|_{L^2}^2 +  C \epsilon^{-1}\left(\|\nabla\theta\|_{L^q}^q+1\right) 
        \end{align}
        Combining \eqref{h3nablaOmega_estimate}, \eqref{h3nablaOmega_estimate_a}, \eqref{h3nablaOmega_estimate_b} \eqref{h3nablaOmega_estimate_part4}, and choosing $\epsilon$ sufficiently small,
        \begin{equation}
            \label{testNablaOmega_estimate}
            \begin{aligned}
                \frac{d}{dt}\|\nabla\omega\|_{L^2}^2 + \nu \|\Delta\omega\|_{L^2}^2 &\lesssim \|u_t\|_{H^1}^2 + (1 + \|u\|_{H^1}^4) \|u\|_{H^2}^2 + \|\nabla\theta\|_{L^q}^q+1.
            \end{aligned}
        \end{equation}
        Taking the gradient of \eqref{advection} and testing with $|\nabla \theta|^{q-2}\nabla\theta$,
        \begin{align}
            \label{gradTheta_estimate}
            \frac{1}{q}\frac{d}{dt}\|\nabla \theta\|_{L^q}^q = - \int_\Omega |\nabla \theta|^{q-2} \nabla\theta \cdot (\nabla \theta \cdot \nabla) u - \int_\Omega |\nabla \theta|^{q-2} \nabla \theta \cdot (u \cdot \nabla) \nabla \theta.
        \end{align}
        Using partial integration, \eqref{incompressibility} and \eqref{nonPenetration} the last on the right-hand side of \eqref{gradTheta_estimate} vanishes as
        \begin{align*}
            0=-\int_\Omega \nabla \cdot u |\nabla\theta|^q = -\int_{\partial\Omega} n\cdot u|\nabla \theta|^q + q \int_\Omega |\nabla \theta|^{q-2}\nabla \theta u: \nabla^2\theta
        \end{align*}
        and therefore Hölder's inequality implies
        \begin{align}
            \label{gradTheta_estimate_2}
            \frac{1}{q}\frac{d}{dt}\|\nabla \theta \|_{L^q}^q &\leq \|\nabla \theta\|_{L^q}^q \|\nabla u\|_{L^\infty}.
        \end{align}
        The logarithmic Sobolev inequality (for details see Lemma 2.2 in \cite{doeringWuZhaoZhen2018})
        \begin{align*}
            \|\nabla u\|_{L^\infty} \leq C \left(1+\|\nabla u \|_{H^1}\right) \log \left(1+\|\nabla u\|_{H^2}\right)
        \end{align*}
        applied to \eqref{gradTheta_estimate_2} and combined with Lemma \ref{preliminaries_lemma_elliptic_regularity} and Young's implies
        \begin{align}
        \notag
                \frac{1}{q}\frac{d}{dt}\|\nabla \theta \|_{L^q}^q 
                &\leq C \|\nabla \theta\|_{L^q}^q \left(1+\|u\|_{H^2}\right)\log \left(1+ C\|\Delta \omega\|_{L^2} + C\|u\|_{H^2}\right)
            \\
                &\leq C \|\nabla \theta\|_{L^q}^q \left(1+\|u\|_{H^2}\right)\log \left(1+ \|\Delta \omega\|_{L^2}\right) + C\|\nabla \theta\|_{L^q}^q (1+\|u\|_{H^2}^2). \label{nabla_theta_qq_estimate}
        \end{align}
        Combining \eqref{testNablaOmega_estimate} and \eqref{nabla_theta_qq_estimate} we find
        \begin{equation}
            \label{yzab_explicitly_written}
            \begin{aligned}
                \frac{d}{dt}\bigg(\|\nabla\omega\|_{L^2}^2 &+ \frac{2}{q}\|\nabla \theta\|_{L^q}^q\bigg) + \nu\|\Delta \omega\|_{L^2}^2
                \\
                &\lesssim \|\nabla\theta\|_{L^q}^q\left(1+\|u\|_{H^2}^2\right) + \|\nabla \theta\|_{L^q}^q \left(1+\|u\|_{H^2}\right)\log \left(1+ \|\Delta \omega\|_{L^2}\right)
                \\
                &\qquad + \|u\|_{H^1}^4 \|u\|_{H^2}^2 + \|u\|_{H^2}^2 + \|u_t\|_{H^1}^2.
            \end{aligned}
        \end{equation}
        Next we define
        \begin{align*}
            Y(t)&= \|\nabla\omega\|_{L^2}^2 + \frac{2}{q}\|\nabla\theta\|_{L^q}^q +1,
            \\
            Z(t)&=\nu\|\Delta\omega\|_{L^2}^2,
            \\
            A(t)&=C\left(1+\|u\|_{H^2}^2+ \|u\|_{H^1}^4 \|u\|_{H^2}^2 + \|u_t\|_{H^1}^2\right),
            \\
            B(t)&= C\left(1+\|u\|_{H^2}\right),
        \end{align*}
        which fulfill $A\in L^1(0,T)$ by  \eqref{H2} and Lemma \ref{lemma_ut_bound}, $B\in L^2(0,T)$ by \eqref{H2} and
        \begin{equation*}
            \begin{aligned}
                \frac{d}{dt}Y(t) + Z(t)
                &\leq A(t)Y(t) + B(t)Y(t)\log\left(1+Z(t)\right)
            \end{aligned}
        \end{equation*}
        because of \eqref{yzab_explicitly_written}. Therefore the logarithmic Grönwall's inequality (for details see Lemma 2.3 in \cite{doeringWuZhaoZhen2018}) implies $Y \in L^\infty(0,T), Z \in L^1(0,T)$
        for all $T>0$, which by their definitions imply
        \begin{align}
            \label{h3_bound_nablaTheta_first_bound}
            \nabla \theta &\in L^\infty\big((0,T);L^q(\Omega)\big)
            \\
            \Delta\omega &\in L^2\left((0,T);L^2(\Omega)\right)
        \end{align}
        for $2<q$.  Since $\|\theta\|_{L^q}$ is bounded by \eqref{theta_conservation} and $\|u\|_{H^2}$ by \eqref{H2}, then by Lemma \ref{preliminaries_lemma_elliptic_regularity}
        \begin{align}
            \label{h3_bound_u_in_h3_in_proof}
            u &\in  L^2\left((0,T);H^3(\Omega)\right).
        \end{align}
        By Hölder's inequality and \eqref{theta_conservation} we are able to generalize \eqref{h3_bound_nablaTheta_first_bound} to
        \begin{align*}
            \theta&\in L^\infty\left((0,T);W^{1,q}(\Omega)\right)
        \end{align*}
        for all $1\leq q <\infty$.
        By Gagliardo-Nirenberg interpolation
        \begin{align*}
            \|u\|_{W^{2,p}}^\frac{2p}{p-2} &\leq C \|u\|_{H^3}^2  \|u\|_{H^2}^{\frac{4}{p-2}}
        \end{align*}
        for all $2<p<\infty$, which because of \eqref{H2} and \eqref{h3_bound_u_in_h3_in_proof} imply $u\in L^\frac{2p}{p-2}\left((0,T);W^{2,p}(\Omega)\right)$
        for all $2\leq p<\infty$, where the case $p = 2$ is due to \eqref{H2}.
\end{itemize}
\end{proof}

\section{Proof of Theorem \ref{new-th-large_time}: Large-Time Behavior}
\label{section_large_time_behavior}

\subsection{Proof of Theorem \ref{new-th-large_time} Part \ref{largetimetheorem_part_decay}}

\paragraph{Argument for (a):}

Recalling \eqref{energy_decay_estimate_ode_uL2}, one has
\begin{align*}
    \frac{d}{dt} \|u\|_{L^2}^2 \leq C \nu^{-1} \|\theta_0\|_{L^2}^2,
\end{align*}
and by \eqref{uInL2H1} $\|u\|_{L^2}^2\in L^1(0,\infty)$. Therefore Lemma \ref{lemma_L1_Nonnegative_UniformlyCont_Implies_Decay} implies
\begin{align}\label{utozero}
    \|u(t)\|_{L^2}^2\to 0 \qquad \text{ for }t\to\infty.
\end{align}

Now we show the convergence to zero of the $H^1$-norm of $u$.   
Testing \eqref{vorticity_eq} with $\omega$
\begin{align}
    \label{test_vorticity}
    \frac{1}{2}\frac{d}{dt}\|\omega\|_{L^2}^2
    &= - \int_\Omega \omega u \cdot \nabla \omega + \nu \int_\Omega \omega \Delta \omega + \int_\Omega \partial_1 \theta \omega.
\end{align}
The first term on the right-hand side of \eqref{test_vorticity} vanishes under integration by parts because of \eqref{incompressibility} and \eqref{nonPenetration}. In order to estimate the second term one needs boundary values for $\nabla \omega$. Notice that by \eqref{nablaOmega_minusDeltaUperp}, $\nabla \omega = -\Delta u^\perp$,
and therefore
\begin{align}
    \label{nablaOmega_along_boundary}
    \nu n\cdot \nabla \omega = \nu \tau \cdot \Delta u = \tau \cdot u_t + \tau \cdot (u\cdot \nabla) u + \tau \cdot \nabla p - \theta \tau_2,
\end{align}
where the second identity is derived from tracing \eqref{navierStokes} along the boundary. Using partial integration, \eqref{vorticity_bc} and \eqref{nablaOmega_along_boundary} we find
\begin{equation}
    \label{vorticity_balance_diffusionTerm}
    \begin{aligned}
        \nu \int_\Omega \omega \Delta \omega
        &= - \nu \|\nabla \omega \|_{L^2}^2 + \nu \int_{\partial\Omega} \omega n \cdot \nabla \omega
        \\
        &=  - \nu \|\nabla \omega \|_{L^2}^2 - 2 \nu \int_{\partial\Omega} (\alpha+\kappa) u_\tau n \cdot \nabla \omega
        \\
        &= - \nu \|\nabla \omega \|_{L^2}^2 - 2 \int_{\partial\Omega} (\alpha+\kappa) u\cdot u_t - 2 \int_{\partial\Omega} (\alpha+\kappa) u \cdot (u\cdot \nabla) u 
        \\
        &\qquad - 2 \int_{\partial\Omega} (\alpha+\kappa) u \cdot \nabla p + 2 \int_{\partial\Omega} (\alpha+\kappa) u_\tau \theta \tau_2
        \\
        &=  - \nu \|\nabla \omega \|_{L^2}^2 - \frac{d}{dt} \int_{\partial\Omega} (\alpha+\kappa) u_\tau^2 - 2 \int_{\partial\Omega} (\alpha+\kappa) u \cdot (u\cdot \nabla) u 
        \\
        &\qquad - 2 \int_{\partial\Omega} (\alpha+\kappa) u \cdot \nabla p - \int_{\partial\Omega} \omega \theta \tau_2.
    \end{aligned}
\end{equation}
For the last term on the right-hand side of \eqref{test_vorticity}, partial integration yields
\begin{align}
    \label{vorticity_balance_temperatureTerm}
    \int_\Omega \partial_1 \theta \omega 
    &= \int_\Omega \nabla \cdot (\theta e_1) \omega
    = \int_{\partial\Omega} n_1 \theta \omega - \int_\Omega \theta \partial_1 \omega.
\end{align}
Combining \eqref{test_vorticity}, \eqref{vorticity_balance_diffusionTerm}, and \eqref{vorticity_balance_temperatureTerm},
\begin{equation}
    \label{vorticity_balance_almost}
    \begin{aligned}
        \frac{1}{2}\frac{d}{dt}\left(\|\omega\|_{L^2}^2 + 2 \int_{\partial\Omega} (\alpha+\kappa) u_\tau^2\right) &+ \nu \|\nabla \omega \|_{L^2}^2
        \\
        &= - 2 \int_{\partial\Omega} (\alpha+\kappa) u \cdot (u\cdot \nabla) u - 2 \int_{\partial\Omega} (\alpha+\kappa) u \cdot \nabla p
        \\
        &\qquad - \int_{\partial\Omega} \omega \theta \tau_2 + \int_{\partial\Omega} n_1 \theta \omega - \int_\Omega \theta \partial_1 \omega.
    \end{aligned}
\end{equation}
As $(-n_2,n_1)=n^\perp=\tau=(\tau_1,\tau_2)$ the third and fourth term on the right-hand side of \eqref{vorticity_balance_almost} cancel resulting in
\begin{equation}
    \label{decayProof_vort_balance}
    \begin{aligned}
        \frac{1}{2}\frac{d}{dt}&\left(\|\omega\|_{L^2}^2 + 2 \int_{\partial\Omega} (\alpha+\kappa) u_\tau^2\right)  + \nu \|\nabla \omega \|_{L^2}^2 
        \\
        &\qquad= - 2 \int_{\partial\Omega} (\alpha+\kappa) u \cdot (u\cdot \nabla) u - 2 \int_{\partial\Omega} (\alpha+\kappa) u \cdot \nabla p - \int_\Omega \theta \partial_1 \omega.
    \end{aligned}
\end{equation}
We estimate the first term on the right-hand side of \eqref{decayProof_vort_balance} by Hölder's inequality and the trace theorem as
\begin{align}
    \label{estimate_ucdotucdotnablau_onboundary}
    \int_{\partial\Omega} \left|(\alpha+\kappa) u\cdot (u\cdot\nabla) u\right|
    &\leq \|\alpha+\kappa\|_{L^\infty} \int_{\partial\Omega} \left|u\cdot (u\cdot\nabla) u\right|
    \leq C \left\|u^2 |\nabla u|\right\|_{W^{1,1}}
    \\
    &\leq C \|u\|_{W^{1,4}}^2 \|u\|_{H^2},
\end{align}
where the constant $C>0$ depends on $\Omega$ and $\alpha$. Due to Lemma \ref{lemma_nablaPcdotU_onBoundary} one gets
\begin{equation}
    \label{estimate_uGradpOnBoundary}
    \begin{aligned}
        -2 \int_{\partial\Omega} (\alpha+\kappa) u\cdot \nabla p 
        \leq C \|p\|_{H^1}\|u\|_{H^1}
    \end{aligned}
\end{equation}
for the third term on the right-hand side of \eqref{decayProof_vort_balance}. The last term on the right-hand side of \eqref{decayProof_vort_balance} can be estimated by
\begin{align}
    \label{estimate_thetaPartial1Omega}
    \int_\Omega |\theta \partial_1\omega |\leq \|\theta\|_{L^2} \| u\|_{H^2}.
\end{align}
Combining \eqref{decayProof_vort_balance}, \eqref{estimate_ucdotucdotnablau_onboundary}, \eqref{estimate_uGradpOnBoundary}, \eqref{estimate_thetaPartial1Omega}, Young's inequality and Lemma \ref{preliminaries_lemma_elliptic_regularity} yields
\begin{equation*}
    \begin{aligned}
        \frac{1}{2}\frac{d}{dt}&\left(\|\omega\|_{L^2}^2 + 2 \int_{\partial\Omega} (\alpha+\kappa) u_\tau^2\right) + \nu \|\nabla \omega \|_{L^2}^2 
        \\
        &\qquad\qquad\qquad\leq C \left(\|u\|_{W^{1,4}}^2 + \|p\|_{H^1} + \|\theta\|_{L^2}\right)\|u\|_{H^2}
        \\
        &\qquad\qquad\qquad\leq C \epsilon^{-1}\left(\|u\|_{W^{1,4}}^4 + \|p\|_{H^1}^2 + \|\theta\|_{L^2}^2\right) + \epsilon C \|\nabla\omega\|_{L^2}^2 + \epsilon C\|u\|_{H^1}^2
    \end{aligned}
\end{equation*}
for all $\epsilon>0$. Choosing $\epsilon = \frac{\nu}{2C}$ and using Lemma \ref{lemma_preliminaries_uv},
\begin{align}
    \label{odeDuPlusAlphaUsquared}
    \frac{d}{dt}\left( \|\D u\|_{L^2}^2 + \int_{\partial\Omega} \alpha u_\tau^2 \right)  + \frac{\nu}{2} \|\nabla \omega\|_{L^2}^2 &\leq C \nu^{-1}\left(\|u\|_{W^{1,4}}^4 + \|p\|_{H^1}^2 + \|\theta\|_{L^2}^2\right) + \nu C\|u\|_{H^1}^2.
\end{align}
Notice by Hölder's inequality, Lemma \ref{lemma_vorticity_bound}, \eqref{theta_conservation}, and Lemma \ref{lemma_pressure_bound}, the right-hand side of \eqref{odeDuPlusAlphaUsquared} is uniformly bounded in time and in particular
\begin{align}
    \frac{d}{dt}\left( \|\D u\|_{L^2}^2 + \int_{\partial\Omega} \alpha u_\tau^2 \right) \leq C
\end{align}
for some constant $C$.
Notice that $t\mapsto \|\D u(t)\|_{L^2}^2 + \int_{\partial\Omega} \alpha u_\tau^2 (t)\in L^1(0,\infty)$ by the trace theorem and \eqref{uInL2H1}.
Therefore Lemma \ref{lemma_L1_Nonnegative_UniformlyCont_Implies_Decay} implies
\begin{align*}
    \|\D u\|_{L^2}^2 + \int_{\partial\Omega} \alpha u_\tau^2 \to 0 \qquad \text{ for }t\to\infty
\end{align*}
and
\begin{align}
    \label{uinH1To0}
    \|u\|_{H^1} \to 0 \qquad \text{ for }t\to\infty
\end{align}
due to Lemma \eqref{lemma_coercivity}. 
Finally the convergence in $W^{1,p}$ for $2\leq p<\infty$ follows by the 
Gagliardo-Nirenberg interpolation inequality
\begin{align}
    \label{uW1pGagliardoNirenberg}
    \|u\|_{W^{1,p}} \leq C \|u\|_{H^2}^{1-\frac{2}{p}} \|u\|_{H^1}^{\frac{2}{p}}
\end{align}
for $2\leq p <\infty$ and the fact that by \eqref{H2} $\|u\|_{H^2}$ is bounded by a constant and that $\|u\|_{H^1}$ vanishes in the time limit by \eqref{uinH1To0}.

\paragraph{Argument for (b):}
Recall that in Lemma \ref{lemma_ut_bound} we proved $u_t \in L^2((0,\infty);H^1(\Omega))$ and
\begin{align*}
    \frac{d}{dt}\|u_t\|_{L^2}^2 &\leq C\,
\end{align*}
(see \eqref{ut_in_L2H1}, \eqref{ut_ode_estimate} and \eqref{ut_l2_bound_in_proof}).
Therefore Lemma \ref{lemma_L1_Nonnegative_UniformlyCont_Implies_Decay} implies
\begin{align}
    \label{utTo0}
    \|u_t(t)\|_{L^2}\to 0 \qquad \text{ for }t\to \infty.
\end{align}

\paragraph{Argument for (c):}

Since $H^1$ is dense in $L^2$, we find for any $g\in L^2$ there exists a function $h\in H^1$ such that $\|g-h\|_{L^2}\leq \epsilon$ for any $\epsilon>0$. Then
\begin{align}
    \langle g, \Delta u \rangle = \langle g-h,\Delta u \rangle + \langle h,\Delta u\rangle.
    \label{deltaU_weak_est_1}
\end{align}
We will estimate these terms individually. By Cauchy-Schwartz inequality
\begin{align}
    |\langle g-h ,\Delta u \rangle |\leq \|g-h\|_{L^2}\|\Delta u \|_{L^2} \leq \|g-h\|_{L^2} \|u\|_{H^2} \leq \epsilon C
    \label{deltaU_weak_est_2}
\end{align}
for some $C>0$ independent of time and any $\epsilon>0$, where in the last estimate we used \eqref{H2}. For the second term on the right-hand side of \eqref{deltaU_weak_est_1}, Lemma \ref{lemma_preliminaries_uv} yields
\begin{align}
    |\langle h,\Delta u\rangle| \leq C \|h\|_{H^1}\|u\|_{H^1}\to 0,
    \label{deltaU_weak_est_3}
\end{align}
thanks to \eqref{uinH1To0}. Combining \eqref{deltaU_weak_est_1}, \eqref{deltaU_weak_est_2}, \eqref{deltaU_weak_est_3} we find that for any $\delta>0$ and $g\in L^2$
\begin{align}
    |\langle g,\Delta u\rangle |\leq \delta
\end{align}
provided $t$ is sufficiently large, i.e.
\begin{align}
    \nabla^\perp \omega = \Delta u\rightharpoonup 0
    \label{DeltaU_weak_to0}
\end{align}
in $L^2$, where we additionally used \eqref{nablaOmega_minusDeltaUperp}.

For any $g\in L^2$, by \eqref{navierStokes} and Hölder's inequality
\begin{align}
    |\langle\nabla p - \theta_2, g\rangle| &= |\langle \nu \Delta u - u_t - u\cdot \nabla u , g\rangle|
    \\
    &\leq \nu |\langle \Delta u, g\rangle| + (\|u\|_{L^4} \|\nabla u\|_{L^4} + \|u_t\|_{L^2}) \|g\|_{L^2}\to 0
\end{align}
thanks to \eqref{DeltaU_weak_to0}, \eqref{largetimetheorem_decay_u} and \eqref{utTo0}, proving \eqref{largetimetheorem_decay_pTheta_weak}.

\begin{remark}\label{about-non-conv-theta}
We note that the weak convergence to a hydrostatic equilibrium state, i.e., $\nabla p(t)-\theta(t) e_2\rightharpoonup 0$, is a property that occurs independently of the limiting behavior of the function $\theta$.
As a side note, we observe that our analysis indicates that $\theta$ does not converge to the linear function $\beta x_2+\gamma$ , which would be the natural equilibrium candidate based on the analysis of the fully dissipative system. In fact, the energy identity \eqref{energy_identity} implies
\begin{equation*}
    \frac{d}{dt} \left( \beta \|u\|_{L^2}^2 + \|\theta-\beta x_2-\gamma\|_{L^2}^2\right) =-4 \beta \nu \left(\|\D u\|_{L^2}^2 + \int_{\partial\Omega} \alpha u_\tau^2 \right) < 0.
\end{equation*}
and, in particular, for all $t\in[0,\infty)$ the function $t\rightarrow \|u\|_{L^2}^2 + \|\theta-\beta x_2-\gamma\|_{L^2}^2 $ is a decreasing function of $t$ and 
$$0< \|u\|_{L^2}^2 + \|\theta-\beta x_2-\gamma\|_{L^2}^2 \leq  C_0^2$$
where $C_0^2=\beta \|u_0\|_{L^2}^2 + \|\theta_0-\beta x_2-\gamma\|_{L^2}^2$.
Hence there exists a constant $c_0$ such that 
$$ \|u\|_{L^2}^2 + \|\theta-\beta x_2-\gamma\|_{L^2}^2 \to c_0^2\leq C_0^2$$
Using the fact that $\|u\|_{L^2}^2\to 0$ as $t\to \infty$ we conclude that 
$$ \|\theta-\beta x_2-\gamma\|_{L^2}^2 \to c_0^2\leq C_0^2$$.
\end{remark}

\subsection{Argument for Theorem \ref{new-th-large_time} Part \ref{largetimetheorem_part_decomposed}}

The Hilbert space $L^2$ can be decomposed (for details see \cite{ConstantinFoiasBook}, Chapter 1) in $L^2 = L^2_\sigma \oplus L^2_\#$, where 
\begin{align}
    L^2_\sigma &= \lbrace \xi \in L^2\ \vert \ \nabla \cdot \xi = 0, n\cdot \xi = 0\rbrace,
    \\
    L^2_\# &= \lbrace \nabla \chi \ \vert \ \chi \in H^1, \int_\Omega \chi = 0\rbrace.    
\end{align}
Let $\xi\in L^2_\sigma$ and $\nabla \chi\in L^2_\#$ such that $\theta e_2 = \xi+\nabla \chi$. Then by \eqref{navierStokes}
\begin{align}
    \|\nabla (\chi-p)\|_{L^2}^2 &= \int_{\Omega} \nabla (\chi-p) \cdot (u_t+ u\cdot \nabla u- \nu \Delta u - \xi)
    \\
    &= \int_{\Omega} \nabla (\chi-p) \cdot (u_t+u\cdot\nabla u) - \nu \int_{\Omega} \nabla (\chi-p) \cdot \Delta u,
    \label{weak_conv_proj_id1}
\end{align}
where the term with $\xi$ vanished due to orthogonality of $L^2_\sigma$ and $L^2_\#$. In order to estimate the last term on the right hand-side of \eqref{weak_conv_proj_id1} we use integration by parts, \eqref{incompressibility} and \eqref{nCdotDeltaU_on_boundary} yield
\begin{align}
    \int_\Omega \nabla (\chi-p) \cdot \Delta u &= \int_{\partial\Omega} (\chi-p) n \cdot \Delta u = 2\int_{\partial\Omega} (\chi-p) \tau\cdot\nabla ((\alpha+\kappa) u_\tau).
    \label{weak_conv_proj_id2}
\end{align}
As connected components of the boundary are periodic and $\tau \cdot\nabla$ is the tangential derivative we find
\begin{align}
    2\int_{\partial\Omega} (\chi-p) \tau\cdot\nabla ((\alpha+\kappa) u_\tau) &= - 2\int_{\partial\Omega} \tau\cdot\nabla (\chi-p) ((\alpha+\kappa) u_\tau)
    \\
    &= - 2\int_{\partial\Omega} (\alpha+\kappa) u \cdot\nabla (\chi-p).
    \label{weak_conv_proj_id3}
\end{align}
Combining \eqref{weak_conv_proj_id1}, \eqref{weak_conv_proj_id2} and \eqref{weak_conv_proj_id3} we find
\begin{align}
    \|\nabla (\chi-p)\|_{L^2}^2 &= \int_{\Omega} \nabla (\chi-p) \cdot (u_t+u\cdot\nabla u) + 2\nu \int_{\partial\Omega} (\alpha+\kappa) u \cdot\nabla (\chi-p)
    \\
    &\leq \|\nabla (\chi -p)\|_{L^2}(\|u_t\|_{L^2} + \|u\|_{W^{1,4}}^2) + C\|u\|_{H^1} \|\chi -p\|_{H^1},
\end{align}
where we additionally used Hölder's inequality and Lemma \ref{lemma_preliminaries_uv}. Since $p$ and $\chi$ can be assumed to be average free Poincaré's inequality yields
\begin{align}
    \|\nabla (\chi-p)\|_{L^2}^2 &\leq \|\nabla (\chi -p)\|_{L^2}(\|u_t\|_{L^2} + \|u\|_{W^{1,4}}^2+C\|u\|_{H^1})
\end{align}
implying
\begin{align}
    \|\nabla (\chi-p)\|_{L^2} &\leq \|u_t\|_{L^2} + \|u\|_{W^{1,4}}^2+C\|u\|_{H^1} \to 0
    \label{oneMinusPthetaMinusGradP_to0}
\end{align}
by \eqref{utTo0} and \eqref{largetimetheorem_decay_u}, proving \eqref{largetimetheorem_decay_p_grad_strong}.

For any $g\in L^2$, using the equation for $u$ we have
\begin{align}
    |\langle \xi, g\rangle| &= |\langle \theta e_2 - \nabla \chi, g \rangle|
    \\
    &= |\langle u_t + u\cdot \nabla u - \nu \Delta u + \nabla p - \nabla \chi, g \rangle|
    \\
    &\leq \nu |\langle \Delta u,g\rangle| + (\|u_t\|_{L^2}+\|u\|_{W^{1,4}}^2 + \|\nabla (\chi-p)\|_{L^2}) \|g\|_{L^2}
    \\
    &\to 0
\end{align}
thanks to \eqref{DeltaU_weak_to0}, \eqref{utTo0}, \eqref{largetimetheorem_decay_u} and \eqref{oneMinusPthetaMinusGradP_to0}, proving \eqref{largetimetheorem_decay_Projtheta_weak}.

Finally by \eqref{navierStokes} and Hölder's inequality
\begin{align}
    \|\nu \Delta u + \xi\|_{L^2} &= \|\Delta u + \theta e_2 - \nabla \chi\|_{L^2}
    \\
    &= \|u_t + u\cdot\nabla u + \nabla p - \nabla \chi\|_{L^2}
    \\
    &\leq \|u_t\|_{L^2} + \|u\|_{W^{1,4}}^2 + \|\nabla (p-\chi)\|_{L^2} \to 0
\end{align}
thanks to \eqref{utTo0}, \eqref{largetimetheorem_decay_u} and \eqref{oneMinusPthetaMinusGradP_to0}, proving \eqref{largetimetheorem_decay_DeltaUprojTheta}.

\subsection{Argument for Theorem \ref{new-th-large_time} Part \ref{largetimetheorem_part_assume}}\hypertarget{proof_convergence_theorem_part_assumption}{The} claims follow easily from the fact that $\xi\rightarrow 0$ strongly in $L^2$. In fact, using the assumption $\lim_{t\rightarrow \infty}\|\theta e_2-\bar{\theta}e_2\|_{L^2}=0$ and that $\langle\xi(t),\nabla \chi(t)\rangle=0$ for all $t$ by orthogonality, then
\begin{align*}
\|\xi\|_{L^2}^2&=\langle\theta e_2-\nabla \chi, \xi\rangle\\
&=\langle\theta e_2-\bar{\theta}e_2, \xi\rangle+\langle\bar{\theta}e_2, \xi\rangle\\
&\leq\|\theta e_2-\bar{\theta}e_2\|_{L^2}\|\xi\|_{L^2}+\langle\bar{\theta}e_2, \xi\rangle\rightarrow 0 \mbox{ as } t\rightarrow \infty
\end{align*}
where we used that, thanks to the weak convergence of $\xi$ to zero (from part \eqref{largetimetheorem_decay_Projtheta_weak}),  $\|\xi\|_{L^2}\leq C$ and $\langle\bar{\theta}e_2, \xi\rangle\rightarrow 0$.

\section{Proof of Theorem \ref{theorem_linearStability}: Linear Stability}
\label{section_linear_stability}
As described in the Introduction, in this section we explore linear stability of the hydrostatic equilibrium, under the following conditions:
\begin{enumerate}
    \item The domain is a periodic strip, $\Omega=\mathbb{T}\times (0,h)$,
    \item The friction coefficient $\alpha$ is constant.
\end{enumerate}
Due to the flat boundaries the curvature vanishes, $n=(0,n_2)$, $\tau = (-n_2,0)$ and $n_2=1$, $n_2=-1$ on the top and bottom boundary respectively.
From these assumptions it follows that the boundary conditions simplify to
\begin{alignat}{2}
    \partial_{2}U_1&=-2\alpha U_1 & \qquad &\mbox{ at } x_2=h,
    \\
    \partial_{2}U_1&=2\alpha U_1 & \qquad &\mbox{ at } x_2=0.
\end{alignat}
Additionally due to \eqref{linearized_pde_theta} and \eqref{linearized_no_penetration} $\Theta$ satisfies $\Theta(t)=\Theta_0$ on the boundaries, implying
\begin{align}
    \partial_1 \Theta &= 0
    \label{linearized_partial1Theta_0_on_boundary}
\end{align}
on $\partial\Omega$.

The corresponding vorticity $\linomega = \nabla^\perp \cdot U$,
analogously to \eqref{vorticity_eq} and \eqref{vorticity_bc}, fulfills
\begin{alignat}{2}
    \linomega_t -\nu\Delta \linomega &= \partial_1\Theta & \qquad\text{ in }&\Omega,\label{lin_vorticity_eq}\\
    \linomega &= 2\alpha U_1 & \qquad\text{ at }&x_2 = h,
    \\
    \linomega &= -2\alpha U_1 & \qquad\text{ at }&x_2 = 0.
\end{alignat}
In order to keep the compact notations of the previous section, for the boundary terms we will write
\begin{alignat}{2}
    n_2\partial_{2}U_1&=-2\alpha U_1 & \qquad &\mbox{ at } \partial\Omega,
    \\
    \linomega&=-2\alpha U_\tau & \qquad &\mbox{ at } \partial\Omega.
    \label{lin_vorticity_bc}
\end{alignat}
In this section we will also use the notation $f\lesssim g$, if there exists a constant $C>0$, potentially depending on $\nu$, $\alpha$ and $\Omega$, such that $f\leq C g$. Notice that similar to Lemma \ref{lemma_L2H1_bound}, Lemma \ref{lemma_ut_bound} and Lemma \ref{lemma_pressure_bound}
\begin{align}
    U&\in L^\infty\left((0,\infty);L^2(\Omega)\right) \cap L^2\left((0,\infty);H^1(\Omega)\right)\label{linearized_u_pre_reg}
    \\
    \Theta &\in L^\infty\left((0,\infty);L^2(\Omega)\right)\label{linearized_theta_pre_reg}
    \\
    U_t&\in L^\infty\left((0,\infty);L^2(\Omega)\right) \cap L^2\left((0,\infty);H^1(\Omega)\right)\label{linearized_ut_pre_reg}
    \\
    \|P\|_{H^1}^2&\lesssim \|U\|_{H^1}^2 + \|\Theta\|_{L^2}^2\label{linearized_p_pre_est}
\end{align}
and using \eqref{linearized_theta_pre_reg} and a corresponding estimate as in Lemma \ref{lemma_energy_bound} $\|U\|_{L^2}^2$ satisfies the assumptions of Lemma \ref{lemma_L1_Nonnegative_UniformlyCont_Implies_Decay}, which implies
\begin{align}
    \label{linearized_uL2_tempTo0}
    \|U\|_{L^2}\to 0.
\end{align}
By Lemma \ref{preliminaries_lemma_elliptic_regularity}, \eqref{nablaOmega_minusDeltaUperp}, \eqref{linearized_pde_u}, \eqref{linearized_p_pre_est}, Gagliardo-Nirenberg interpolation and Young's inequality
\begin{equation*}
    \begin{aligned}
        \|U\|_{H^2}^2 &\lesssim \|\nabla\linomega\|_{L^2}^2 +\|U\|_{H^1}^2\lesssim \|\Delta U\|_{L^2}^2 + \|U\|_{H^1}^2\lesssim \|U\|_{H^1}^2 + \|P\|_{H^1}^2 +\|\Theta\|_{L^2}^2 + \|U_t\|_{L^2}^2 
        \\
        &\lesssim  \|U\|_{H^1}^2 +\|\Theta\|_{L^2}^2 + \|U_t\|_{L^2}^2 \lesssim \epsilon \|U\|_{H^2}^2+\left(1+\epsilon^{-1}\right)\|U\|_{L^2}^2 +\|\Theta\|_{L^2}^2 + \|U_t\|_{L^2}^2 
    \end{aligned}
\end{equation*}
for any $\epsilon >0$. Choosing $\epsilon$ sufficiently small we can absorb the $H^2$ term on the right-hand side and using \eqref{linearized_u_pre_reg}, \eqref{linearized_theta_pre_reg} and \eqref{linearized_ut_pre_reg} find
\begin{align}
    \label{linearized_uH2_tempreg}
    U\in L^\infty\left((0,\infty);H^2(\Omega)\right).
\end{align}
Again using Gagliardo-Nirenberg interpolation we get $H^1$-decay
\begin{align*}
    \|U\|_{H^1}^2 \lesssim \|U\|_{H^2}\|U\|_{L^2} \to 0,
\end{align*}
where we used \eqref{linearized_uH2_tempreg} and \eqref{linearized_uL2_tempTo0}. Combining these estimates we get the regularity results
\begin{align}
    U&\in L^\infty\left((0,\infty);H^2(\Omega)\right) \cap L^2\left((0,\infty);H^1(\Omega)\right)\label{linearized_u_reg}
    \\
    \Theta &\in L^\infty\left((0,\infty);L^2(\Omega)\right)\label{linearized_theta_reg}
    \\
    U_t&\in L^\infty\left((0,\infty);L^2(\Omega)\right) \cap L^2\left((0,\infty);H^1(\Omega)\right)\label{linearized_ut_reg}
    \\
    P &\in L^\infty\left((0,\infty);H^1(\Omega)\right)\label{linearized_p_reg}
\end{align}
and
\begin{align}
    \label{linearized_u_H1_to_0}
    \|U(t)\|_{H^1}&\to 0
\end{align}
for $t\to\infty$ similar to the nonlinear system.%

\begin{proof}[\hypertarget{proof_theorem_linearStability}{Proof of Theorem \ref{theorem_linearStability}.}]

First, in order to show the regularity of $U_t$ we need to derive bounds for $P_t$. In order to do so notice that taking the divergence and time derivative of \eqref{linearized_pde_u} and using the incompressibility and \eqref{linearized_pde_theta}
\begin{align}
    \label{linearized_pde_laplace_P}
    \Delta P_t = \partial_2 \Theta_t = - \beta \partial_2 U_2.
\end{align}
Using integration by parts and \eqref{linearized_pde_laplace_P}
\begin{align}
    \label{linearzed_pt_id}
    \|\nabla P_t\|_{L^2}^2 = \int_{\partial\Omega} P_t n\cdot \nabla P_t - \int_{\Omega} P_t \Delta P_t = \int_{\partial\Omega} P_t n\cdot \nabla P_t + \beta\int_{\Omega} P_t \partial_2 U_2.
\end{align}
In order to calculate the boundary term notice that similar to \eqref{pressure_bc} $P_t$ fulfills
\begin{align}
    \label{linearized_pt_bc}
    n\cdot \nabla P_t = 2\nu \tau\cdot \nabla (\alpha U_t\cdot \tau) + n_2 \Theta_t = 2\nu \tau\cdot \nabla (\alpha U_t\cdot \tau) - \beta n_2 U_2.
\end{align}
Plugging \eqref{linearized_pt_bc} into \eqref{linearzed_pt_id} and using the periodicity of the boundary we find
\begin{equation}
    \label{linearzed_pt_id_2}
    \begin{aligned}
        \|\nabla P_t\|_{L^2}^2 &= 2\nu  \int_{\partial\Omega} P_t \tau\cdot \nabla (\alpha U_t\cdot \tau) - \beta \int_{\partial\Omega} P_t n_2 U_2 + \beta\int_{\Omega} P_t \partial_2 U_2
        \\
        &= -2\nu  \int_{\partial\Omega} \alpha U_t \cdot \nabla  P_t  - \beta \int_{\partial\Omega} P_t n_2 U_2 + \beta\int_{\Omega} P_t \partial_2 U_2
        \\
        &\lesssim \|P_t\|_{H^1}\|U_t\|_{H^1} + \|P_t\|_{H^1}\|U\|_{H^1},
    \end{aligned}
\end{equation}
where in the last inequality we used Lemma \ref{lemma_nablaPcdotU_onBoundary}, trace Theorem and Hölder's inequality. As $P$ is only defined up to a constant we can choose it to be average free, which also implies that $P_t$ has vanishing average. Therefore Poincaré's inequality and \eqref{linearzed_pt_id_2} imply
\begin{align}
    \label{linearzed_pt_bound} 
    \|P_t\|_{H^1}^2 \lesssim \|\nabla P_t\|_{L^2}^2 \lesssim \|P_t\|_{H^1}\|U_t\|_{H^1} + \|P_t\|_{H^1}\|U\|_{H^1},
\end{align}
which after dividing by $\|P\|_{H^1}$ yields $\|P_t\|_{H^1} \lesssim \|U_t\|_{H^1} + \|U\|_{H^1}$.

With a bound for the time derivative of the pressure at hand we can estimate $\linomega_t$. Taking the time derivative of \eqref{lin_vorticity_eq}, testing with $\linomega_t$ and using \eqref{linearized_pde_theta} we find
\begin{align}
    \label{linearzed_omegat_id_1}
    \frac{1}{2}\frac{d}{dt}\|\linomega_t\|_{L^2}^2 = \nu \int_{\Omega} \linomega_t \Delta \linomega_t + \int_{\Omega} \linomega_t \partial_1 \Theta_t = \nu \int_{\Omega} \linomega_t \Delta \linomega_t - \beta \int_{\Omega} \linomega_t \partial_1 U_2.
\end{align}
We first focus on the first term on the right-hand side of \eqref{linearzed_omegat_id_1}. Partial integration yields
\begin{equation}
    \label{linearzed_omegat_id_2}
    \begin{aligned}
        \nu\int_{\Omega} \linomega_t \Delta \linomega_t
        &= \nu \int_{\partial\Omega} \linomega_t n\cdot \nabla \linomega_t - \nu\|\nabla \linomega_t\|_{L^2}^2.
    \end{aligned}
\end{equation}
Similar to \eqref{nablaOmega_along_boundary} we find
\begin{align}
    \label{linearzed_omegat_id_3}
    \nu n\cdot \nabla \linomega_t = \tau\cdot U_{tt} + \tau \cdot \nabla P_t - \tau_2 \Theta_t = \tau\cdot U_{tt} + \tau \cdot \nabla P_t + \beta \tau_2 U_2.
\end{align}
Plugging \eqref{linearzed_omegat_id_2} and \eqref{linearzed_omegat_id_3} into \eqref{linearzed_omegat_id_1} and, since by \eqref{lin_vorticity_bc}, $\linomega_t = -2 \alpha \tau \cdot U_t$,
\begin{equation*}
    \begin{aligned}
        \frac{1}{2}\frac{d}{dt}&\|\linomega_t\|_{L^2}^2 + \nu\|\nabla \linomega_t\|_{L^2}^2
        \\
        &= \int_{\partial\Omega} \linomega_t \tau\cdot U_{tt} + \int_{\partial\Omega} \linomega_t \tau \cdot \nabla P_t + \beta \int_{\partial\Omega} \linomega_t\tau_2 U_2 - \beta \int_{\Omega} \linomega_t \partial_1 U_2
        \\
        &= -2\int_{\partial\Omega} \alpha\tau\cdot U_t\tau\cdot U_{tt} - 2 \int_{\partial\Omega} \alpha U_t \cdot \nabla P_t + \beta \int_{\partial\Omega} \linomega_t\tau_2 U_2 - \beta \int_{\Omega} \linomega_t \partial_1 U_2
        \\
        &= - \frac{d}{dt}\int_{\partial\Omega} \alpha(\tau\cdot U_t)^2 - 2 \int_{\partial\Omega} \alpha U_t \cdot \nabla P_t + \beta \int_{\partial\Omega} \linomega_t\tau_2 U_2 - \beta \int_{\Omega} \linomega_t \partial_1 U_2,
    \end{aligned}
\end{equation*}
which we can estimate by Lemma \ref{lemma_nablaPcdotU_onBoundary}, trace theorem and Hölder's inequality as
\begin{equation*}
    \begin{aligned}
        \frac{1}{2}\frac{d}{dt}&\left(\|\linomega_t\|_{L^2}^2 +2 \int_{\partial\Omega} \alpha(\tau\cdot U_t)^2\right)+ \nu\|\nabla \linomega_t\|_{L^2}^2
        \\
        &\qquad\qquad\lesssim \|U_t\|_{H^1}\|P_t\|_{H^1} + \|U_t\|_{H^2}\|U\|_{H^1} + \|U_t\|_{H^1}\|U\|_{H^1}.
    \end{aligned}
\end{equation*}
\eqref{linearzed_pt_bound} and Young's inequality yield
\begin{equation*}
    \begin{aligned}
        \frac{1}{2}\frac{d}{dt}&\left(\|\linomega_t\|_{L^2}^2 +2 \int_{\partial\Omega} \alpha(\tau\cdot U_t)^2\right)+ \nu\|\nabla \linomega_t\|_{L^2}^2
        \\
        &\qquad\qquad\lesssim \|U_t\|_{H^1}^2 + \|U_t\|_{H^2}\|U\|_{H^1} + \|U_t\|_{H^1}\|U\|_{H^1}
        \\
        &\qquad\qquad\lesssim \|U_t\|_{H^1}^2 + \epsilon\|U_t\|_{H^2}^2 + (1+\epsilon^{-1})\|U\|_{H^1}^2
    \end{aligned}
\end{equation*}
for any $\epsilon>0$. Next by Lemmas \ref{lemma_preliminaries_uv} and \ref{preliminaries_lemma_elliptic_regularity} and choosing $\epsilon$ sufficiently small,
\begin{equation}
    \begin{aligned}
        \frac{d}{dt}\left(\|\D U_t\|_{L^2}^2 + \int_{\partial\Omega} \alpha(\tau\cdot U_t)^2\right) + \nu\|U_t\|_{H^2}^2 &\lesssim \|U_t\|_{H^1}^2 + \epsilon\|U_t\|_{H^2}^2 + (1+\epsilon^{-1})\|U\|_{H^1}^2
        \\
        \label{linearzed_omegat_id_4}
    \frac{d}{dt}\left(\|\D U_t\|_{L^2}^2 +  \int_{\partial\Omega} \alpha(\tau\cdot U_t)^2\right)+ \nu\|U_t\|_{H^2}^2 &\lesssim \|U_t\|_{H^1}^2 + \|U\|_{H^1}^2.
    \end{aligned}
\end{equation}
Integrating \eqref{linearzed_omegat_id_4} in time and noticing that the right-hand side is uniformly bounded in time by \eqref{linearized_ut_reg} and \eqref{linearized_u_reg} we find
\begin{align}
    \label{linearized_Ut_L2H2}
    U_t\in L^2\left((0,\infty);H^2(\Omega)\right)
\end{align}
as the bracket on the left-hand side of \eqref{linearzed_omegat_id_4} is positive since $\alpha>0$. In fact the equivalent statement to \eqref{lemma_coercivity} yields
\begin{align*}
    \|U_t\|_{H^1}^2 \lesssim \|\D U_t\|_{L^2}^2 + \int \alpha (\tau \cdot U_t)^2,
\end{align*}
implying
\begin{align}
    \label{linearized_Ut_LinftH1}
    U_t \in L^\infty\left((0,\infty);H^1(\Omega)\right).
\end{align}

Using \eqref{linearzed_omegat_id_4} we find that $\|\D U_t\|_{L^2}^2+\int \alpha (\tau \cdot U_t)^2$ satisfies Lemma \ref{lemma_L1_Nonnegative_UniformlyCont_Implies_Decay}, which implies
\begin{align*}
    \|\D U_t\|_{L^2}^2+\int \alpha (\tau \cdot U_t)^2 \to 0.
\end{align*}
Therefore Lemma \ref{lemma_coercivity} yields
\begin{align}
    \label{linearized_Ut_to_0}
    \|U_t\|_{H^1} \to 0
\end{align}
for $T\to \infty$.
Next, we show that $\|\nabla\Theta\|_{L^2}$ is uniformly bounded in time, which will be used to prove the decay of $\|U\|_{H^2}$. In order to do so we first prove the following two identities
\begin{align}
    \frac{1}{2} \frac{d}{dt} \left(\|\linomega\|_{L^2}^2+ \frac{1}{\beta}\|\nabla \Theta\|_{L^2}^2 \right) + \nu \|\nabla\linomega\|_{L^2}^2 &= - 2\nu \int_{\partial\Omega} \alpha U\cdot\Delta U,
    \label{linearized_id_A}
    \\
    \hspace{-5pt}\frac{1}{2}\frac{d}{dt}\left(\|\nabla \linomega\|_{L^2}^2 + \frac{1}{\beta} \|\nabla^2 \Theta\|_{L^2}^2 + \frac{2}{\beta} \int_{\partial\Omega} \alpha (\partial_2 \Theta)^2\right) + \nu \|\Delta \linomega\|_{L^2}^2 &= - 2 \nu \int_{\partial\Omega} \alpha U_t\cdot \Delta U.
    \label{linearized_id_B}
\end{align}
We first show \eqref{linearized_id_A}: taking the gradient of \eqref{linearized_pde_theta} one has
\begin{align}
	\nabla \Theta_t + \beta \nabla U_2 &= 0.
	\label{linearized_nablaTheta_pde}
\end{align}
Testing \eqref{linearized_nablaTheta_pde} with $\frac{1}{\beta} \nabla \Theta$ and \eqref{lin_vorticity_eq} with $\linomega$ and adding them we obtain
\begin{align}
	\frac{1}{2} \frac{d}{dt} \left(\|\linomega\|_{L^2}^2+ \frac{1}{\beta}\|\nabla \Theta\|_{L^2}^2 \right) &= \nu \int_{\Omega} \linomega \Delta \linomega + \int_{\Omega} \linomega \partial_1 \Theta -  \int_{\Omega} \nabla \Theta \cdot \nabla U_2.
	\label{linearized_H1_id_a}
\end{align}
We first focus on the last two terms on the right-hand side \eqref{linearized_H1_id_a}. By the definition of $\linomega$ and \eqref{linearized_incompressible}
\begin{equation}
    \label{linearized_H1_id_b}
    \begin{aligned}
    	\int_{\Omega} \linomega \partial_1 \Theta -  \int_{\Omega} \nabla \Theta \cdot \nabla U_2 &= \int_{\Omega} (-\partial_2 U_1 + \partial_1 U_2) \partial_1 \Theta - \int_{\Omega} \partial_1 \Theta \partial_1 U_2 - \int_\Omega \partial_2 \Theta \partial_2 U_2
    	\\
    	&= - \int_{\Omega} \partial_2 U_1 \partial_1 \Theta - \int_\Omega \partial_2 U_2 \partial_2 \Theta 
    	\\
    	&= - \int_{\Omega} \partial_2 U_1 \partial_1 \Theta + \int_\Omega \partial_1 U_1 \partial_2 \Theta.        
    \end{aligned}
\end{equation}
Integrating by parts twice, using \eqref{linearized_partial1Theta_0_on_boundary} and the periodicity in the horizontal direction, these terms cancel as
\begin{equation}
    \label{linearized_H1_id_c}
    \begin{aligned}
        - \int_{\Omega} \partial_2 U_1 \partial_1 \Theta + \int_\Omega \partial_1 U_1 \partial_2 \Theta &= - \int_{\partial\Omega} n_2 U_1 \partial_1 \Theta + \int_{\Omega} U_1 \partial_1\partial_2 \Theta + \int_{\Omega} \partial_1 U_1 \partial_2 \Theta
        \\
        &= - \int_{\Omega} \partial_1 U_1 \partial_2 \Theta + \int_{\Omega} \partial_1 U_1 \partial_2 \Theta = 0.
    \end{aligned}
\end{equation}
In order to treat the first term on the right-hand side of \eqref{linearized_H1_id_a} notice that by \eqref{nablaOmega_minusDeltaUperp}
\begin{align}
    n \cdot \nabla \linomega = - n \cdot \Delta U^\perp = \tau \cdot \Delta U
    \label{linearized_nnablaomega_minustauDeltaU}
\end{align}
and therefore, using integration by parts, \eqref{lin_vorticity_bc} and \eqref{nablaOmega_minusDeltaUperp}, it follows that
\begin{equation}
    \label{linearized_H1_id_d}
    \begin{aligned}
        \nu \int_\Omega \linomega \Delta \linomega
        &= - \nu \|\nabla \linomega \|_{L^2}^2 + \nu \int_{\partial\Omega} \linomega n \cdot \nabla \linomega
        =  - \nu \|\nabla \linomega \|_{L^2}^2 - 2 \nu \int_{\partial\Omega} \alpha U_\tau n \cdot \nabla \linomega
        \\
        &=  - \nu \|\nabla \linomega \|_{L^2}^2 - 2 \nu \int_{\partial\Omega} \alpha U_\tau \tau \cdot \Delta U
        %
        =  - \nu \|\nabla \linomega \|_{L^2}^2 - 2 \nu \int_{\partial\Omega} \alpha U \cdot \Delta U.
    \end{aligned}
\end{equation}
Therefore combining \eqref{linearized_H1_id_a}, \eqref{linearized_H1_id_b}, \eqref{linearized_H1_id_c} and \eqref{linearized_H1_id_d} results in
\begin{align*}
    \frac{1}{2} \frac{d}{dt} \left(\|\linomega\|_{L^2}^2+ \frac{1}{\beta}\|\nabla \Theta\|_{L^2}^2 \right) + \nu \|\nabla\linomega\|_{L^2}^2 &= - 2\nu \int_{\partial\Omega} \alpha U\cdot\Delta U,
\end{align*}
proving the first identity, i.e. \eqref{linearized_id_A}. 

Now we show \eqref{linearized_id_B}: taking the gradient of \eqref{lin_vorticity_eq} one has
\begin{align}
    \nabla \linomega_t - \nu \nabla \Delta \linomega &= \nabla \partial_1 \Theta
    \label{linearized_nabla_omega_pde}
\end{align}
and similar for \eqref{linearized_nablaTheta_pde}
\begin{align}
    \nabla^2 \Theta_t + \beta \nabla^2 U_2 &= 0.
    \label{linearized_hessian_theta_pde}
\end{align}
By \eqref{linearized_nabla_omega_pde}, \eqref{linearized_hessian_theta_pde} and \eqref{linearized_nablaTheta_pde} one obtains
\begin{equation}
    \label{linearized_H2_id_a}
    \begin{aligned}
        \frac{1}{2}\frac{d}{dt}&\left(\|\nabla \linomega\|_{L^2}^2 + \frac{1}{\beta} \|\nabla^2 \Theta\|_{L^2}^2 + \frac{2}{\beta} \int_{\partial\Omega} \alpha (\partial_2 \Theta)^2\right) 
        \\
        &= \nu \int_{\Omega} \nabla \linomega \cdot\nabla \Delta \linomega + \int_{\Omega} \nabla \linomega \cdot \nabla \partial_1 \Theta - \int_{\Omega} \nabla^2\Theta \colon \nabla^2 U_2 - 2 \int_{\partial\Omega} \alpha \partial_2\Theta \partial_2 U_2.
    \end{aligned}
\end{equation}
Using integration by parts, \eqref{linearized_nnablaomega_minustauDeltaU} and \eqref{lin_vorticity_eq}, the first term on the right-hand side of \eqref{linearized_H2_id_a} satisfies
\begin{equation}
    \label{linearized_H2_id_b}
    \begin{aligned}
        \nu \int_{\Omega} \nabla \linomega \cdot\nabla \Delta \linomega &= - \nu \|\Delta \linomega \|_{L^2}^2 + \nu \int_{\partial\Omega} \nabla \linomega \cdot n \Delta \linomega
        \\
        &= - \nu \|\Delta \linomega \|_{L^2}^2 + \nu \int_{\partial\Omega} \tau \cdot \Delta U \Delta \linomega
        \\
        &= - \nu \|\Delta \linomega \|_{L^2}^2 + \nu \int_{\partial\Omega} \tau \cdot \Delta U (\linomega_t-\partial_1 \Theta)
        \\
        &= - \nu \|\Delta \linomega \|_{L^2}^2 + \nu \int_{\partial\Omega} \tau \cdot \Delta U \linomega_t,
    \end{aligned}
\end{equation}
where in the last identity we used that $\partial_1 \Theta$ vanishes on $\partial\Omega$ by \eqref{linearized_partial1Theta_0_on_boundary}. Deriving \eqref{lin_vorticity_bc} with respect to time one has
\begin{align*}
    \linomega_t = -2\alpha \tau \cdot U_t
\end{align*}
on $\partial\Omega$, which combined with \eqref{linearized_H2_id_b} implies
\begin{align}
    \nu \int_{\Omega} \nabla \linomega \cdot\nabla \Delta \linomega &= - \nu \|\Delta \linomega \|_{L^2}^2 - 2 \nu \int_{\partial\Omega} \alpha U_t\cdot \Delta U.
    \label{linearized_H2_id_b2}
\end{align}
Next we will show that the last three terms on the right-hand side of \eqref{linearized_H2_id_a} cancel. By the definition of $\linomega$
\begin{equation}
    \label{linearized_H2_id_c}
    \begin{aligned}
        \hspace{-5pt}\int_{\Omega} \nabla \linomega \cdot \nabla \partial_1 \Theta - \int_{\Omega} \nabla^2\Theta \colon \nabla^2 U_2 
        &= \int_{\Omega} \nabla (-\partial_2 U_1+\partial_1 U_2) \cdot \nabla \partial_1 \Theta - \int_{\Omega} \nabla^2\Theta \colon \nabla^2 U_2
        \\
        &= -\int_{\Omega} \nabla \partial_2 U_1 \cdot \nabla \partial_1 \Theta - \int_{\Omega} \nabla \partial_2 \Theta \cdot \nabla \partial_2 U_2
        \\
        &= -\int_{\Omega} \nabla \partial_2 U_1 \cdot \nabla \partial_1 \Theta + \int_{\Omega} \nabla \partial_2 \Theta \cdot \nabla \partial_1 U_1,
    \end{aligned}
\end{equation}
where in the last identity we used \eqref{linearized_incompressible}. Integrating by parts twice, using the periodicity of the domain and \eqref{linearized_partial1Theta_0_on_boundary}, yields
\begin{align*}
    -\int_{\Omega} \partial_1 \partial_2 U_1 \partial_1^2 \Theta &= \int_{\Omega} \partial_1^2 \partial_2 U_1 \partial_1 \Theta = \int_{\partial\Omega} n_2 \partial_1^2 U_1 \partial_1 \Theta - \int_{\Omega} \partial_1^2 U_1 \partial_1 \partial_2 \Theta
    \\
    &= - \int_{\Omega} \partial_1^2 U_1 \partial_1 \partial_2 \Theta,
\end{align*}
which combined with \eqref{linearized_H2_id_c} implies
\begin{align}
    \int_{\Omega} \nabla \linomega \cdot \nabla \partial_1 \Theta - \int_{\Omega} \nabla^2\Theta \colon \nabla^2 U_2 
    &= -\int_{\Omega} \partial_2^2 U_1 \partial_1 \partial_2 \Theta + \int_{\Omega} \partial_2^2 \Theta \partial_1 \partial_2 U_1.
    \label{linearized_H2_id_d}
\end{align}
Using integration by parts for both terms, the periodicity of the domain and that by \eqref{linearized_navier_slip} $U$ satisfies
\begin{align*}
    n_2 \partial_1 \partial_2 U_1 + 2 \alpha \partial_1 U_1 = 0
\end{align*}
on $\partial\Omega$, we find
\begin{equation}
    \label{linearized_H2_id_e}
    \begin{aligned}
        -\int_{\Omega} \partial_2^2 U_1 \partial_1 \partial_2 \Theta &+ \int_{\Omega} \partial_2^2 \Theta \partial_1 \partial_2 U_1
        \\
        &=  \int_{\Omega}  \partial_1\partial_2^2 U_1 \partial_2 \Theta + \int_{\partial\Omega} n_2\partial_2 \Theta \partial_1 \partial_2 U_1 - \int_{\Omega} \partial_2 \Theta \partial_1 \partial_2^2 U_1
        \\
        &=  \int_{\Omega}  \partial_1\partial_2^2 U_1 \partial_2 \Theta - 2 \int_{\partial\Omega} \alpha\partial_2 \Theta \partial_1 U_1 - \int_{\Omega} \partial_2 \Theta \partial_1 \partial_2^2 U_1
        \\
        &= 2 \int_{\partial\Omega} \alpha\partial_2 \Theta \partial_2 U_2,
    \end{aligned}
\end{equation}
where in the last identity we used \eqref{linearized_incompressible}. Therefore combining \eqref{linearized_H2_id_d} and \eqref{linearized_H2_id_e} results in
\begin{equation}
     \label{linearized_H2_id_f}
    \begin{aligned}
         \int_{\Omega} \nabla \linomega \cdot \nabla \partial_1 \Theta &- \int_{\Omega} \nabla^2\Theta \colon \nabla^2 U_2 - 2 \int_{\partial\Omega} \alpha \partial_2\Theta \partial_2 U_2
         \\
         &=-\int_{\Omega} \partial_2^2 U_1 \partial_1 \partial_2 \Theta + \int_{\Omega} \partial_2^2 \Theta \partial_1 \partial_2 U_1 - 2 \int_{\partial\Omega} \alpha \partial_2\Theta \partial_2 U_2
         \\
         &=  2 \int_{\partial\Omega} \alpha\partial_2 \Theta \partial_2 U_2 - 2 \int_{\partial\Omega} \alpha \partial_2\Theta \partial_2 U_2 = 0,
    \end{aligned}
\end{equation}
i.e. the last three terms on the right-hand side of \eqref{linearized_H2_id_a} cancel. Therefore by \eqref{linearized_H2_id_a}, \eqref{linearized_H2_id_b2} and \eqref{linearized_H2_id_f} one has
\begin{align*}
    \hspace{-5pt}\frac{1}{2}\frac{d}{dt}&\left(\|\nabla \linomega\|_{L^2}^2 + \frac{1}{\beta} \|\nabla^2 \Theta\|_{L^2}^2 + \frac{2}{\beta} \int_{\partial\Omega} \alpha (\partial_2 \Theta)^2\right) + \nu \|\Delta \linomega\|_{L^2}^2 = - 2 \nu \int_{\partial\Omega} \alpha U_t\cdot \Delta U,
\end{align*}
proving the second identity, i.e. \eqref{linearized_id_B}.

Combining \eqref{linearized_id_A} and \eqref{linearized_id_B} we obtain
\begin{equation}
    \label{linearized_combination_a}
    \begin{aligned}
        \frac{1}{2}\frac{d}{dt}&\left(\|\nabla \linomega\|_{L^2}^2 + \|\linomega\|_{L^2}^2 + \frac{1}{\beta} \|\nabla^2 \Theta\|_{L^2}^2 + \frac{1}{\beta}\|\nabla \Theta\|_{L^2}^2  + \frac{2}{\beta} \int_{\partial\Omega} \alpha (\partial_2 \Theta)^2\right)
        \\
        &\qquad+ \nu \|\Delta \linomega\|_{L^2}^2 + \nu \|\nabla\linomega\|_{L^2}^2
        \\
        &= - 2 \nu \int_{\partial\Omega} \alpha (U+U_t)\cdot \Delta U
    \end{aligned}
\end{equation}
and, using trace estimate and Young's inequality, we have
\begin{equation}
     \label{linearized_combination_b}
    \begin{aligned}
         \nu \int_{\partial\Omega} |\alpha (U+U_t)\cdot \Delta U| &\leq \alpha \nu C (\|U\|_{H^1}+\|U_t\|_{H^1})\|U\|_{H^3}
         \\
         &\leq \nu \epsilon \|U\|_{H^3}^2 + \epsilon^{-1} \nu \alpha^2 C (\|U\|_{H^1}^2 + \|U_t\|_{H^1}^2)
    \end{aligned}
\end{equation}
for any $\epsilon>0$. By Lemma \ref{preliminaries_lemma_elliptic_regularity}, proven in the appendix, there exist constants $C,C_\alpha>0$ such that
\begin{align}
    \|U\|_{H^3}^2 \leq C \|\Delta \linomega\|_{L^2}^2 + C_\alpha \|U\|_{H^2}^2 \leq C \|\Delta \linomega\|_{L^2}^2 + C_\alpha \|\nabla \linomega\|_{L^2}^2 + C_\alpha \|U\|_{H^1}^2,
    \label{linearized_combinatino_H3_estimate}
\end{align}
which combined with \eqref{linearized_combination_a} and \eqref{linearized_combination_b} yields
\begin{equation}
    \label{linearized_combination_c}
    \begin{aligned}
        \frac{1}{2}\frac{d}{dt}&\left(\|\nabla \linomega\|_{L^2}^2 + \|\linomega\|_{L^2}^2 + \frac{1}{\beta} \|\nabla^2 \Theta\|_{L^2}^2 + \frac{1}{\beta}\|\nabla \Theta\|_{L^2}^2  + \frac{2}{\beta} \int_{\partial\Omega} \alpha (\partial_2 \Theta)^2\right)
        \\
        &\qquad+ \nu \|\Delta \linomega\|_{L^2}^2 + \nu \|\nabla\linomega\|_{L^2}^2
        \\
        &\leq C\epsilon \nu \|\Delta \linomega\|_{L^2}^2 + C_\alpha \epsilon \nu \|\nabla\linomega\|_{L^2}^2 + \tilde C_\alpha (1+\epsilon^{-1})\nu (\|U\|_{H^1}^2+\|U_t\|_{H^1}^2).
    \end{aligned}
\end{equation}
Choosing $\epsilon = \frac{1}{2(C+C_\alpha)}$ the $\linomega$ terms on the right-hand side of \eqref{linearized_combination_c} can be absorbed and therefore
\begin{equation}
    \label{linearized_combination_d}
    \begin{aligned}
        \frac{d}{dt}&\left(\|\nabla \linomega\|_{L^2}^2 + \|\linomega\|_{L^2}^2 + \frac{1}{\beta} \|\nabla^2 \Theta\|_{L^2}^2 + \frac{1}{\beta}\|\nabla \Theta\|_{L^2}^2  + \frac{2}{\beta} \int_{\partial\Omega} \alpha (\partial_2 \Theta)^2\right)
        \\
        &\qquad+ \nu \|\Delta \linomega\|_{L^2}^2 + \nu \|\nabla\linomega\|_{L^2}^2
        \\
        &\lesssim \|U\|_{H^1}^2+\|U_t\|_{H^1}^2
    \end{aligned}
\end{equation}
holds. Integrating in time $t\in(0,T)$ for any $T>0$ and using that the right-hand side of \eqref{linearized_combination_d} is integrable in time by \eqref{linearized_u_reg} and \eqref{linearized_ut_reg}, we obtain
\begin{align*}
    &\|\nabla \linomega(T)\|_{L^2}^2 + \|\linomega(T)\|_{L^2}^2 + \frac{1}{\beta} \|\nabla^2 \Theta(T)\|_{L^2}^2 + \frac{1}{\beta}\|\nabla \Theta(T)\|_{L^2}^2  + \frac{2}{\beta} \int_{\partial\Omega} \alpha (\partial_2 \Theta(T))^2
    \\
    &\quad\lesssim \frac{1}{\beta} \|U_0\|_{H^2}^2 + \|\Theta_0\|_{H^2}^2 + \frac{2}{\beta} \int_{\partial\Omega} \alpha (\partial_2\theta_0)^2 + \int_0^\infty (\|U(s)\|_{H^1}^2 + \|U_t(s)\|_{H^1}^2)\ ds
    \\
    &\quad\leq C
\end{align*}
for some constant $C>0$, independent of time, and hence,
\begin{align}
    \Theta \in L^\infty\left((0,\infty);H^2(\Omega)\right).
    \label{linearized_theta_in_H2}
\end{align}
Next we show the decay of $\|U\|_{H^2}$.
By \eqref{linearized_id_A}, \eqref{linearized_pde_u} and \eqref{linearized_nablaTheta_pde}
\begin{align*}
    \nu \|\nabla\linomega\|_{L^2}^2 &= - 2\nu \int_{\partial\Omega} \alpha U\cdot\Delta U - \frac{1}{2} \frac{d}{dt} \left(\|\linomega\|_{L^2}^2+ \frac{1}{\beta}\|\nabla \Theta\|_{L^2}^2 \right)
    \\
    &= - 2\nu \int_{\partial\Omega} \alpha U\cdot \Delta U - \int_\Omega \linomega \linomega_t - \frac{1}{\beta} \int_\Omega \nabla \Theta \cdot\nabla \Theta_t
    \\
    &= - 2\nu \int_{\partial\Omega} \alpha U\cdot(U_t + \nabla P - \Theta e_2) - \int_\Omega \linomega \linomega_t + \int_\Omega \nabla \Theta \cdot\nabla U_2
\end{align*}
and using trace theorem, Lemma \ref{lemma_nablaPcdotU_onBoundary} and Hölder's inequality
\begin{align*}
    \nu \|\nabla\linomega\|_{L^2}^2 \lesssim \|U\|_{H^1} (\|U_t\|_{H^1} + \|P\|_{H^1} + \|\Theta\|_{H^1}).
\end{align*}
By \eqref{linearized_Ut_LinftH1}, \eqref{linearized_p_reg} and \eqref{linearized_theta_in_H2} the bracket is universally bounded in time and as $\|U\|_{H^1}$ decays by \eqref{linearized_u_H1_to_0}, Lemma \ref{preliminaries_lemma_elliptic_regularity} results in
\begin{align}
    \|U(t)\|_{H^2}^2 \lesssim \|\nabla\linomega(t)\|_{L^2}^2 + \|U(t)\|_{H^1}^2 \to 0
    \label{linearized_U_H2_to_0}
\end{align}
for $t\to \infty$.

 By \eqref{linearized_pde_u}
\begin{align*}
    \|\nabla P-\Theta e_2\|_{L^2} = \|U_t - \nu \Delta U\|_{L^2} \leq \|U_t\|_{L^2} + \nu\|U\|_{H^2} \to 0,
\end{align*}
for $T\to \infty$, where in the last limit we used \eqref{linearized_Ut_to_0} and \eqref{linearized_U_H2_to_0}. This proves the desired convergence.

Finally we prove a higher order regularity. 
As $U,U_t \in L^2\left((0,\infty);H^1(\Omega)\right)$ by \eqref{linearized_u_reg} and \eqref{linearized_ut_reg}, integrating \eqref{linearized_combination_d} in time shows that
\begin{align*}
    \nu &\int_0^\infty (\|\Delta \linomega(s)\|_{L^2}^2 + \|\nabla \linomega(s)\|_{L^2}^2)\ ds
    \\
    &\quad\lesssim \frac{1}{\beta} \|U_0\|_{H^2}^2 + \|\Theta_0\|_{H^2}^2 + \frac{2}{\beta} \int_{\partial\Omega} \alpha (\partial_2\theta_0)^2 + \int_0^\infty (\|U(s)\|_{H^1}^2 + \|U_t(s)\|_{H^1}^2)\ ds
    \\
    &\quad<\infty
\end{align*}
and therefore \eqref{linearized_combinatino_H3_estimate} implies
\begin{align}
    U\in L^2\left((0,\infty);H^3(\Omega)\right).
\end{align}

\end{proof}

\section{Appendix}

\subsection{Gradient Estimates}\hypertarget{subsection:GradientEst}{\phantom{a}}

Here we provide key lemmas that enable our analysis. Lemma \ref{lemma_preliminaries_uv} proves an equivalence of norms for the full gradient, strain tensor, and vorticity.

In the following we use the notation
\begin{align*}
    \int \nabla u:\nabla v\ dx &=\int \partial_iu_j\partial_i v_j \ dx
    \\
    \int \D u:\D v\ dx &= \int (\D u)_{ij}(\D v)_{ij} \ dx =\frac{1}{2}\int \partial_i u_j\partial_i v_j+\partial_i u_j\partial_j u_i\ dx.
\end{align*}

\begin{lemma}
\label{lemma_preliminaries_uv}
Assume $\Omega$ is a $C^{1,1}$-domain, $u, v\in H^1(\Omega)$ fulfill \eqref{nonPenetration} and $u$ satisfies \eqref{incompressibility}. Then for $\omega = \nabla^\perp \cdot u$ and $\eta = \nabla^\perp \cdot v$, where $\nabla^\perp = (-\partial_2,\partial_1)$,
\begin{equation}
    \label{lemma_uv_identities_first_id}
    \begin{aligned}
        2\int_\Omega \D u: \D v
        = \int_\Omega \nabla u: \nabla v + \int_{\partial\Omega} \kappa u_\tau v_\tau
        = \int_\Omega \omega \eta + 2 \int_{\partial\Omega} \kappa u_\tau v_\tau.
    \end{aligned}
\end{equation}
If additionally $u\in H^2$ fulfills \eqref{navSlip}
\begin{equation}
    \label{lemma_uv_identities_second_id}
    \begin{aligned}
        -\int_\Omega \Delta u\cdot v
        &= 2\int_\Omega \D u: \D v + 2\int_{\partial\Omega} \alpha u_\tau v_\tau
    \end{aligned}
\end{equation}
and
\begin{equation}
    \label{lemma_uv_identities_third_id}
    \begin{aligned}
        \left|\int_\Omega \Delta u\cdot \Xi\right|
        &\leq C \|u\|_{H^1}\| \Xi\|_{H^1}
    \end{aligned}
\end{equation}
for all $\Xi\in H^1$, where $C=C(\alpha,\Omega)>0$.
\end{lemma}
\begin{proof}
Assume at first $u\in H^2$. Then
\begin{align}
    \label{DuDv_estimate_1}
    2 \int_\Omega \D u: \D v = \frac{1}{2} \int_\Omega (\nabla u + \nabla u^T) : (\nabla v + \nabla v^T) = \int_\Omega \nabla u : \nabla v + \int_\Omega \nabla u : \nabla v^T.
\end{align}
Using integration by parts the second term in \eqref{DuDv_estimate_1} is given by
\begin{align}
    \label{DuDv_estimate_2}
    \int_\Omega \partial_i u_j \partial_j v_i
    = - \int_\Omega \partial_i\partial_j u_j v_i + \int_{\partial\Omega} n \cdot (v\cdot \nabla) u
    = \int_{\partial\Omega} v_\tau n\cdot (\tau\cdot\nabla) u,
\end{align}
where the last identity is due to $u$ satisfying \eqref{incompressibility} and $v$ satisfying \eqref{nonPenetration}. Combining \eqref{DuDv_estimate_1}, \eqref{DuDv_estimate_2} and \eqref{vorticity_bc_derivation_2} we find
\begin{align*}
    2 \int_\Omega \D u: \D v = \int_\Omega \nabla u : \nabla v + \int_{\partial\Omega} \kappa u_\tau v_\tau,
\end{align*}
which by density also holds for $u\in H^1$, proving the first identity of \eqref{lemma_uv_identities_first_id}.

Similar to the proof of the first identity, integration by parts implies
\begin{align}
    \label{DuDv_estimate_a_1}
    \int_\Omega \nabla u : \nabla v
    = - \int_\Omega v \cdot \Delta u + \int_{\partial\Omega} v \cdot (n\cdot \nabla) u
    = - \int_\Omega v \cdot \nabla^\perp \omega + \int_{\partial\Omega} v_\tau \tau \cdot (n\cdot \nabla) u,
\end{align}
where the last identity is due to \eqref{nablaOmega_minusDeltaUperp} and $v$ satisfies \eqref{nonPenetration}. Using integration by parts again the first term on the right-hand side of \eqref{DuDv_estimate_a_1} is given by
\begin{align}
    \label{DuDv_estimate_a_2}
    -\int_\Omega v\cdot \nabla^\perp \omega
    = \int_\Omega v^\perp \cdot \nabla \omega
    = \int_{\partial\Omega} v^\perp \cdot n \omega - \int_\Omega \nabla \cdot v^\perp \omega
    = - \int_{\partial\Omega} v_\tau \omega + \int_\Omega \eta \omega.
\end{align}
In order to calculate the remaining boundary terms notice that the algebraic identity $\tau_i \tau_j + n_i n_j = \delta_{ij}$, where $\delta_{ij}$ is the Kronecker delta, i.e. $\delta_{ij}=0$ for $i\neq j$ and $\delta_{ij}=1$ for $i=j$, holds and therefore
\begin{equation*}
    \begin{aligned}
        \tau \cdot (n\cdot\nabla) u - \omega
        &= \tau \cdot (n\cdot\nabla) u - \nabla^\perp \cdot u
        = \tau \cdot (n\cdot\nabla) u - \tau \cdot (\tau \cdot \nabla^\perp) u - n \cdot (n \cdot \nabla^\perp) u
        \\
        &= \tau \cdot (n\cdot\nabla) u + \tau \cdot (\tau^\perp \cdot \nabla) u + n \cdot (n^\perp \cdot \nabla) u
        = n \cdot (\tau \cdot \nabla) u = \kappa u_\tau,        
    \end{aligned}
\end{equation*}
where the identity is due to \eqref{vorticity_bc_derivation_2}. Combining \eqref{DuDv_estimate_a_1}, \eqref{DuDv_estimate_a_2} and \eqref{DuDv_estimate_a_2} yields the second identity in \eqref{lemma_uv_identities_first_id}.

To prove \eqref{lemma_uv_identities_second_id}, integration by parts yields
\begin{align}
    \label{DuDv_estimate_b_1}
    -\int_\Omega \Delta u \cdot v = - \int_{\partial\Omega} v \cdot (n \cdot \nabla) u + \int_\Omega \nabla u : \nabla v = - \int_{\partial\Omega} v_\tau \tau \cdot (n \cdot \nabla) u + \int_\Omega \nabla u : \nabla v, 
\end{align}
as $v$ satisfies \eqref{nonPenetration}. The boundary term can be written as
\begin{align}
    \label{DuDv_estimate_b_2}
    -\tau \cdot (n\cdot\nabla) u = - 2 (\D u \ n)\cdot \tau + n \cdot (\tau\cdot \nabla ) u = (2 \alpha + \kappa) u_\tau,
\end{align}
where in the last identity we used \eqref{navSlip} and \eqref{vorticity_bc_derivation_2}. Combining \eqref{DuDv_estimate_b_1}, \eqref{DuDv_estimate_b_2} and \eqref{lemma_uv_identities_first_id} yields the claim.

Next we focus on \eqref{lemma_uv_identities_third_id}. Partial integration again yields
\begin{align*}
    -\int_{\Omega} \Delta u \cdot \Xi = \int_\Omega \nabla u \colon \nabla \Xi - \int_{\partial\Omega} \Xi \cdot (n \cdot \nabla) u
\end{align*}
Note that $n_i n_j + \tau_i \tau_j = \delta_{ij}$ and \eqref{DuDv_estimate_b_2} imply
\begin{equation}
    \label{lemma_uv_third_proof_est1}
    \begin{aligned}
        -\int_{\Omega} \Delta u \cdot \Xi &= \int_\Omega \nabla u \colon \nabla \Xi - \int_{\partial\Omega} (\Xi\cdot \tau) \tau \cdot (n \cdot \nabla) u - \int_{\partial\Omega} (\Xi\cdot n) n \cdot (n \cdot \nabla) u
        \\
        &= \int_\Omega \nabla u \colon \nabla \Xi + \int_{\partial\Omega} (2\alpha+\kappa) \Xi_\tau u_\tau - \int_{\partial\Omega} (\Xi\cdot n) n \cdot (n \cdot \nabla) u.
    \end{aligned}
\end{equation}
In order to estimate the last term on the right-hand side of \eqref{lemma_uv_third_proof_est1} we extend $n$ to the whole space $\Omega$ according to Corollary \ref{corollary_boundary_extension}, i.e. there exists $\eta\in W^{1,4}(\Omega)$
fulfilling
\begin{align}
    \eta\vert_{\partial\Omega} &= n,\\
    \label{lemma_uv_third_proof_eta_bound}
    \|\eta\|_{W^{1,4}} &\leq C \|n\|_{W^{1,\infty}} \leq C (1 + \|\kappa\|_\infty).
\end{align}
Therefore Stokes' theorem implies
\begin{equation}
    \label{lemma_uv_third_proof_est2}
    \begin{aligned}
        \bigg|&\int_{\partial\Omega} (\Xi\cdot n) n \cdot (n \cdot \nabla) u\bigg|
        \\
        &= \left| \int_{\partial\Omega} n \cdot \left( (\Xi\cdot \eta)  (\eta \cdot \nabla) u \right) \right|
        = \left| \int_{\Omega} \nabla \cdot \left( (\Xi\cdot \eta)  (\eta \cdot \nabla) u \right) \right|
        \\
        &\leq \left| \int_{\Omega} \partial_i \Xi_j \eta_j \eta_k \partial_k u_i \right| + \left| \int_{\Omega} \Xi_j \partial_i\eta_j \eta_k \partial_k u_i \right| + \left| \int_{\Omega} \Xi_j \eta_j \partial_i \eta_k \partial_k u_i \right| + \left| \int_{\Omega} \Xi_j \eta_j \eta_k \partial_k \partial_i u_i \right|
        \\
        &\leq \|\eta\|_{L^\infty} \left(\|\eta\|_{L^\infty} \|\nabla\Xi\|_{L^2} \|\nabla u\|_{L^2} + 2 \|\nabla \eta\|_{L^4} \|\Xi\|_{L^4} \|\nabla u\|_{L^2} \right) + \left| \int_{\Omega} (\Xi \cdot \eta) (\eta \cdot \nabla) (\nabla \cdot u) \right|.
    \end{aligned}
\end{equation}
The last term in \eqref{lemma_uv_third_proof_est2} vanishes by \eqref{incompressibility} and the remaining terms can be bounded using Sobolev embedding resulting in
\begin{equation}
    \label{lemma_uv_third_proof_est3}
    \begin{aligned}
        \left|\int_{\partial\Omega} (\Xi\cdot n) n \cdot (n \cdot \nabla) u\right|
        &\leq C\|\eta\|_{W^{1,4}}^2 \|\Xi\|_{H^1} \|\nabla u\|_{L^2} \leq C (1+\|\kappa\|_\infty^2) \|\Xi\|_{H^1} \|\nabla u\|_{L^2},
    \end{aligned}
\end{equation}
where in the last inequality we used \eqref{lemma_uv_third_proof_eta_bound}.
Using Hölder's inequality for the first and trace Theorem and Hölder's inequality for the second term on the right-hand side of \eqref{lemma_uv_third_proof_est1} we find, after plugging in \eqref{lemma_uv_third_proof_est3},
\begin{align*}
    \left|\int_{\Omega} \Delta u \cdot \Xi \right| &\leq \left(1+ \|2\alpha+\kappa\|_{L^\infty} + \|\kappa\|_{L^\infty}^2 \right) C \|\Xi\|_{H^1} \|u\|_{H^1}.
\end{align*}
\end{proof}

Lemma \ref{lemma_coercivity} exploits Lemma \ref{lemma_preliminaries_uv}, showing coercivity for the Laplace operator combined with the Navier-slip boundary conditions due to $\alpha>0$.

\begin{lemma}[Coercivity]
\label{lemma_coercivity}
Let $\Omega$ be a $C^{1,1}$-domain, $\alpha>0$ almost everywhere on $\partial\Omega$ and $u\in H^1$ satisfy \eqref{incompressibility} and \eqref{nonPenetration}. Then there exists a constant $C = C(|\Omega|)>0$ 
such that
\begin{align}
    \|\nabla u\|_{L^2}^2 + \int_{\partial\Omega} \alpha u_\tau^2&\geq \frac{1}{C} \min \left\lbrace 1,\|\alpha^{-1}\|_{L^\infty}^{-1}\right\rbrace \|u\|_{H^1}^2,
    \label{coercitvity_estimate_nabla}
    \\
    \|\D u\|_{L^2}^2 + \int_{\partial\Omega} \alpha u_\tau^2&\geq \frac{1}{C} \min\left\lbrace 1, \|\alpha^{-1}\|_{L^\infty}^{-1}, \|\alpha^{-1}\kappa\|_{L^\infty}^{-1}\right\rbrace \|u\|_{H^1}^2.
    \label{coercitvity_estimate_D}
\end{align}
\end{lemma}

\begin{proof}
First use the fundamental theorem of calculus, Young's and Hölder's inequality to get
\begin{align}
    \label{coercivity_proof_estimate_u_pointwise}
    |u(x_1,x_2)|^2 &= \left|u(x_1,\tilde x_2)+ \int_{\tilde x_2}^{x_2} \partial_2 u(x_1,z) \ dz\right|^2 \leq 2 u_\tau^2(x_1,\tilde x_2) + 2|x_2-\tilde x_2|\|\partial_2 u\|_{L^2_v(x_1)}^2
\end{align}
for all $x\in \Omega$ and $\tilde x_2$ such that $(x_1,\tilde x_2)\in \partial\Omega$ and the straight line from $(x_1,\tilde x_2)$ to $(x_1,x_2)$ are completely in $\Omega$, where $L^2_v(x_1)$ indicates the $L^2$-norm of the vertical line pieces at $x_1$. In the inequality in \eqref{coercivity_proof_estimate_u_pointwise} we also used $u(x_1,\tilde x_2)=u_\tau(x_1,\tilde x_2)$ as $u\cdot n=0$ on the boundary by \eqref{nonPenetration}. Integrating over $\Omega$ we find
\begin{align}
    \label{poincare_type_estimate}
    \|u\|_{L^2}^2 \leq C \left(\int_{\partial\Omega} u_\tau^2+\|\nabla u\|_{L^2}^2\right) \leq C \max\left\lbrace 1, \| \alpha^{-1}\|_{L^\infty} \right\rbrace \left( \int_{\partial\Omega} \alpha u_\tau^2 + \|\nabla u\|_{L^2}^2 \right)
\end{align}
for some constant $C>0$ depending on $|\Omega|$, proving \eqref{coercitvity_estimate_nabla}. In order to prove \eqref{coercitvity_estimate_D} notice that by Lemma \ref{lemma_preliminaries_uv}
\begin{align*}
    \|\nabla u\|_{L^2}^2 &= 2\|\D u\|_{L^2}^2 - \int_{\partial\Omega} \kappa u_\tau^2 \leq 2\|\D u\|_{L^2}^2 + \|\alpha^{-1}\kappa\|_{L^\infty} \int_{\partial\Omega} \alpha u_\tau^2 
    \\
    &\leq 2\max \left\lbrace 1,  \|\alpha^{-1}\kappa\|_{L^\infty}\right\rbrace \left(\|\D u\|_{L^2}^2+ \frac{1}{2}\int_{\partial\Omega} \alpha u_\tau^2\right),
\end{align*}
which yields
\begin{align*}
    \|\D u\|_{L^2}^2 + \int_{\partial\Omega} \alpha  u_\tau^2 
    &\geq \frac{1}{2}\min\left\lbrace 1, \|\alpha^{-1}\kappa\|_{L^\infty}^{-1}\right\rbrace \|\nabla u\|_{L^2}^2 + \frac{1}{2}\int_{\partial\Omega}\alpha u_\tau^2 
    \\
    &\geq \frac{1}{2}\min\left\lbrace 1, \|\alpha^{-1}\|_{L^\infty}^{-1}, \|\alpha^{-1}\kappa\|_{L^\infty}^{-1}\right\rbrace \left( \|\nabla u\|_{L^2}^2 + \int_{\partial\Omega} u_\tau^2 \right)
    \\
    &\geq \frac{1}{C} \min\left\lbrace 1, \|\alpha^{-1}\|_{L^\infty}^{-1}, \|\alpha^{-1}\kappa\|_{L^\infty}^{-1}\right\rbrace \|u\|_{H^1}^2,
\end{align*}
where in the last inequality we used \eqref{poincare_type_estimate}.
\end{proof}

Using the stream function formulation, Lemma \ref{preliminaries_lemma_elliptic_regularity} extends Lemma \ref{lemma_preliminaries_uv}, allowing us to exchange between the full gradient and the vorticity for higher order Sobolev norms.

\begin{lemma}
\label{preliminaries_lemma_elliptic_regularity}
Let $1<p<\infty$, $k\in \mathbb{N}_0$, $1\leq r\leq \infty$, $\Omega$ be a $C^{k+1,1}$-domain and $u$ satisfy \eqref{incompressibility} and \eqref{nonPenetration} and $\omega=\nabla^\perp \cdot u$. Then there exists a constant $C=C(\Omega,p,k,r)>0$ 
such that
\begin{align}
    \label{ellptic_reg_first_part_inequality}
    \|\nabla u\|_{W^{k,p}} \leq C \left(\|\omega\|_{W^{k,p}} + \|u\|_{L^r}\right).
\end{align}
Additionally for $\alpha\in W^{k+2,\infty}(\partial\Omega)$ and a $C^{k+3,1}$-domain $\Omega$, 
\begin{align}
    \label{ellptic_reg_second_part_inequality}
    \|\nabla u\|_{W^{k+2,p}}\leq C \|\Delta \omega\|_{W^{k,p}}+ C\left(\|\alpha+\kappa\|_{W^{k+2,\infty}(\partial\Omega)} + 1\right) \|u\|_{W^{k+2,p}}.
\end{align}
\end{lemma}
\begin{proof}
Let $\phi$ be the stream function of $u$, i.e. $u=\nabla^\perp \phi$, then taking the curl of $u$ shows $\Delta \phi = \omega$. Additionally $\phi$ is constant along connected components of the boundary as 
\begin{align*}
    \frac{d}{d\lambda}\phi(x_1(\lambda),x_2(\lambda)) = \frac{d}{d\lambda}x(\lambda) \cdot \nabla\phi=\tau\cdot \nabla \phi = \tau^\perp \cdot \nabla^\perp \phi = - n \cdot u = 0,
\end{align*}
where $\lambda$ is the parameterization of $\partial\Omega$ by arc length. Therefore the stream function fulfills
\begin{alignat*}{2}
        \Delta \phi &= \omega &\qquad \text{ in }&\Omega
        \\
        \phi &= \psi_i &\qquad \text{ on }&\Gamma_i,
\end{alignat*}
where $\psi_i$ are constants and $\Gamma_i$ are the connected components of $\partial\Omega$. As $\phi$ is only defined up to a constant we can assume that it has vanishing average to be able to use Poincaré inequality. By elliptic regularity (for details see Remark 2.5.1.2 in \cite{grisvard1985}) $\phi\mapsto (-\Delta \phi,\phi\vert_{\partial\Omega})$ is an isomorphism from $W^{k+2,p}(\Omega)$ onto $W^{k,p}(\Omega)\times W^{k+2-\frac{1}{p},p}(\partial\Omega)$, which combined with the definition of $\phi$ implies
\begin{align}
    \label{elliptic_reg_est1}
    \|\nabla u\|_{W^{k,p}} \leq \|\phi\|_{W^{k+2,p}}\leq C \|\omega\|_{W^{k,p}} + C \|\phi\|_{W^{k+2-\frac{1}{p},p}(\partial\Omega)}.
\end{align}
 In order to estimate the boundary terms note that by trace estimate
\begin{equation}
    \label{elliptic_reg_est2}
    \begin{aligned}
        \|\phi\|_{W^{s,p}(\partial\Omega)}^p
        &= \sum_i \left\| \psi_i \right\|_{W^{s,p}(\Gamma_i)}^p
        = \sum_i \int_{\Gamma_i} |\psi_i|^p
        = \|\phi\|_{L^p(\partial\Omega)}^p
        \\
        &\leq \|\phi\|_{W^{1,p}(\Omega)}^p
        \leq C \|\nabla \phi\|_{L^p(\Omega)}^p
        = C \|u\|_{L^p}^p        
    \end{aligned}
\end{equation}
for all $s\geq 0$, where in the last inequality we used Poincaré inequality and in the last identity the definition of $\phi$. Gagliardo-Nirenberg interpolation and Young's inequality imply
\begin{align}
    \label{elliptic_reg_est3}
    \|u\|_{L^p} \leq C \|\nabla u\|_{L^p}^\rho \|u\|_{L^1}^{1-\rho} + C \|u\|_{L^1} \leq \epsilon \rho C \|\nabla u\|_{L^p} + \left((1-\rho) \epsilon^{-\frac{\rho}{1-\rho}} + 1\right) C \|u\|_{L^1}
\end{align}
for all $\epsilon>0$, where $\rho = \frac{2p-2}{3p-2}$. Combining \eqref{elliptic_reg_est1}, 
\eqref{elliptic_reg_est2}, \eqref{elliptic_reg_est3} we find
\begin{align*}
    \|\nabla u\|_{W^{k,p}} \leq C \|\omega\|_{W^{k,p}} + \epsilon C \|\nabla u\|_{L^p} + C_\epsilon \|u\|_{L^r},
\end{align*}
due to H\"older's inequality applied to the last term.
This estimate yields \eqref{ellptic_reg_first_part_inequality} after choosing $\epsilon$ sufficiently small.

The proof of \eqref{ellptic_reg_second_part_inequality} follows a similar strategy. By elliptic regularity for the vorticity, $\omega \mapsto (-\Delta \omega, \omega\vert_{\partial\Omega})$ is an isomorphism from $W^{k+2,p}(\Omega)$ onto $W^{k,p}(\Omega)\times W^{k+2-\frac{1}{p},p}(\partial\Omega)$, which combined with the boundary condition \eqref{vorticity_bc} yields
\begin{align}
    \label{higher_order_ellpitic_regularity_est_1}
    \|\omega\|_{W^{k+2,p}} \leq C \|\Delta \omega\|_{W^{k,p}} + C\|(\alpha+\kappa)u_\tau\|_{W^{k+2-\frac{1}{p},p}(\partial\Omega)}.
\end{align}
The boundary term can be estimated by Hölder's inequality and the trace theorem (for details see Theorem 1.5.1.2 in \cite{grisvard1985}) by
\begin{equation}
    \label{higher_order_ellpitic_regularity_est_2}
    \begin{aligned}
        \|(\alpha+\kappa) u_\tau\|_{W^{k+2-\frac{1}{p},p}(\partial\Omega)}
        &\leq \|\alpha+\kappa\|_{W^{k+2,\infty}(\partial\Omega)} \|u\|_{W^{k+2,p}(\Omega)}.        
    \end{aligned}
\end{equation}
Combining \eqref{ellptic_reg_first_part_inequality}, 
\eqref{higher_order_ellpitic_regularity_est_1} and \eqref{higher_order_ellpitic_regularity_est_2} we find
\begin{align*}
    \|\nabla u\|_{W^{k+2,p}}
    \leq C \|\Delta \omega\|_{W^{k,p}} + C\left(\|\alpha+\kappa\|_{W^{k+2,\infty}(\partial\Omega)}+1\right) \|u\|_{W^{k+2,p}}.
\end{align*}

\end{proof}

Lemma \ref{lemma_nablaPcdotU_onBoundary} provides estimates for gradients multiplied with vector fields satisfying \eqref{incompressibility} and \eqref{nonPenetration}. While integrals over the whole space would vanish because of orthogonality, they are in general nonzero for boundary integrals. However, the orthogonality provides bounds that require less regularity than trace estimates.

\begin{lemma}
\label{lemma_nablaPcdotU_onBoundary}
For a $C^{1,1}$-domain $\Omega$, $f\in W^{1,\infty}(\partial\Omega)$, $\rho \in H^1(\Omega)$ and $v\in H^1(\Omega)$ satisfying $\nabla \cdot v = 0$ in $\Omega$ and $v\cdot n = 0$ on $\partial\Omega$ there exists a constant $C>0$ depending only on the Lipschitz constant of $\Omega$ such that
\begin{align*}
    \left|\int_{\partial\Omega}  f v \cdot \nabla \rho \right| = \left|\int_{\partial\Omega} \rho \tau\cdot \nabla (f v\cdot \tau) \right| \leq C \|fn\|_{W^{1,\infty}(\partial\Omega)} \|\rho\|_{H^1} \|v\|_{H^1}.
\end{align*}
\end{lemma}

\begin{proof}
First, as connected components of the boundary are closed curves and $\tau \cdot \nabla$ is the tangential derivative along $\partial\Omega$ parameterized by arc-length, we find the integration by parts formula
\begin{align*}
    \int_{\partial\Omega} f v \cdot \nabla \rho
    = \int_{\partial\Omega} f v_\tau \tau \cdot \nabla \rho
    = \int_{\partial\Omega} \tau \cdot \nabla \left(f v_\tau \rho\right) - \int_{\partial\Omega} \rho \tau\cdot \nabla (f v_\tau)
    = - \int_{\partial\Omega} \rho \tau\cdot \nabla (f v_\tau),
\end{align*}
where $v_\tau = v\cdot \tau$. Using product rule and $v\cdot n=0$
\begin{align}
    \label{proof_nablaPcdotU_onBoundary_id_1}
    \int_{\partial\Omega} \rho \tau \cdot \nabla \left(f v_\tau \right) = \int_{\partial\Omega} \rho v \cdot \nabla f + \int_{\partial\Omega} f \rho \tau \cdot \nabla v_\tau.
\end{align}
The first term on the right-hand side of \eqref{proof_nablaPcdotU_onBoundary_id_1} can be bounded using Hölder's inequality and the trace theorem as
\begin{align*}
    \left|\int_{\partial\Omega} \rho v \cdot \nabla f\right|\leq C \|f\|_{W^{1,\infty}(\partial\Omega)} \|\rho\|_{H^1}\|v\|_{H^1} \leq C \|fn\|_{W^{1,\infty}(\partial\Omega)} \|\rho\|_{H^1}\|v\|_{H^1}.
\end{align*}
In order to estimate the second term on the right-hand side of \eqref{proof_nablaPcdotU_onBoundary_id_1} notice that since $W^{1,\infty}(\partial\Omega)\subset W^{1-\frac{1}{p},p}(\partial\Omega)$ 
for all $1<p< \infty$ we are able to extend $f n \in W^{1,\infty}(\partial\Omega)$ to $\eta \in W^{1,p}(\Omega)$ as described in Corollary \ref{corollary_boundary_extension}, i.e.
\begin{align}
    \label{proof_nablaPcdotU_onBoundary_id_2}
    \eta\vert_{\partial\Omega} = fn, \qquad \|\eta\|_{W^{1,p}(\Omega)} \leq C \|f n\|_{W^{1,\infty}(\partial\Omega)}.
\end{align}
Since $\tau \cdot (\tau \cdot \nabla) \tau = \frac{1}{2}\tau \cdot \nabla (\tau\cdot \tau) = 0$
\begin{align*}
    \tau\cdot \nabla v_\tau &= \tau \cdot \nabla (v\cdot \tau) = v\cdot (\tau \cdot \nabla) \tau + \tau \cdot (\tau \cdot \nabla) v =  v_\tau \tau \cdot (\tau \cdot \nabla) \tau + \tau \cdot (\tau \cdot \nabla) v
    \\
    &= \tau \cdot (\tau \cdot \nabla) v,
\end{align*}
which combined with $\tau_i \tau_j+ n_i n_j = \delta_{ij}$ implies
\begin{align}
    \label{proof_nablaPcdotU_onBoundary_id_3}
    \tau\cdot \nabla v_\tau = \tau_i \tau_j \partial_i v_j = \delta_{ij} \partial_i v_j - n_i n_j \partial_i v_j = \nabla \cdot v - n \cdot (n \cdot \nabla) v = n \cdot (n \cdot \nabla) v,
\end{align}
where in the last identity we used the divergence free assumption on $v$. Using \eqref{proof_nablaPcdotU_onBoundary_id_2} and \eqref{proof_nablaPcdotU_onBoundary_id_3} and Stokes' theorem we can estimate the second term on the right-hand side of \eqref{proof_nablaPcdotU_onBoundary_id_1} as
\begin{equation}
    \label{proof_nablaPcdotU_onBoundary_id_4}
    \begin{aligned}
        \int_{\partial\Omega} f \rho \tau \cdot \nabla v_\tau
        &= - \int_{\partial\Omega} f \rho n \cdot (n \cdot \nabla) v
        = - \int_{\partial\Omega} \rho n \cdot (\eta \cdot \nabla) v
        = - \int_{\Omega} \nabla \cdot (\rho (\eta \cdot \nabla) v)
        \\
        &= - \int_{\Omega} \nabla \rho \cdot (\eta \cdot \nabla) v - \int_{\Omega} \rho \partial_i \eta_j \partial_j v_i - \int_{\Omega} \rho (\eta \cdot \nabla) (\nabla \cdot v)
    \end{aligned}
\end{equation}
The last term on the right-hand side of \eqref{proof_nablaPcdotU_onBoundary_id_4} vanishes by the divergence free assumption on $v$ and for the other terms we use Hölder's inequality, the Sobolev embedding theorem, and \eqref{proof_nablaPcdotU_onBoundary_id_2} to find
\begin{align*}
    \left| \int_{\partial\Omega} f \rho \tau \cdot \nabla v_\tau \right|
    &\leq C \|\eta\|_{L^\infty} \|\rho\|_{H^1} \|v\|_{H^1} + C\|\eta\|_{W^{1,4}} \|\rho\|_{L^4} \|v\|_{H^1}
    \\
    &\leq C \|\eta\|_{W^{1,4}} \|\rho\|_{H^1} \|v\|_{H^1}
    \\
    &\leq C \| f n\|_{W^{1,\infty}(\partial\Omega)} \|\rho\|_{H^1} \|v\|_{H^1}.
\end{align*}

\end{proof}

\subsection{Technical Lemmas}\hypertarget{subsection:technicalLemmas}{\phantom{a}}

The following Lemma is a slight modification of Lemma 3.1 in \cite{doeringWuZhaoZhen2018}. For details see Lemma 49 in \cite{bleitner2024Buoyancy}.
\begin{lemma}
\label{lemma_L1_Nonnegative_UniformlyCont_Implies_Decay}
Assume $f\in L^1(0,\infty)$ is a non negative that satisfies $f'\leq C$ for a constant $C$ and all $t\geq 0$. Then
\begin{align*}
    f(t)\to 0 \qquad \text{ for } t\to\infty.
\end{align*}
\end{lemma}
The following Corollary is an immediate consequence of Theorem 1.5.1.2 in \cite{grisvard1985} with $l=k=0$ and $s=1$.
\begin{corollary}
\label{corollary_boundary_extension}
Let $\Omega$ be a $C^{1,1}$-domain and $1<p<\infty$. Then for every $g\in W^{1-\frac{1}{p},p}(\partial\Omega)$ there exists $f\in W^{1,p}(\Omega)$ with
\begin{align*}
    f\vert_{\partial\Omega} = g\qquad \text{ and }\qquad \|f\|_{W^{1,p}(\Omega)}\leq C \|g\|_{W^{1-\frac{1}{p},p}(\partial\Omega)}.
\end{align*}
\end{corollary}

\section*{Acknowledgments}

E.C. acknowledges and respects the Lekwungen peoples on whose traditional territory the University of Victoria stands, and the Songhees, Esquimalt and WSÁNEĆ peoples whose historical relationships with the land continue to this day.
C.N. thanks Christian Seis for useful comments on a previous version of the paper.

\paragraph{Funding}

F.B. was supported by DFG-GRK-2583.
The research of E.C. was supported by the Pacific Institute for the Mathematical Sciences (PIMS). The research and findings may not reflect those of the Institute. E.C. was also supported in part by the Department of Defense Vannevar Bush Faculty Fellowship, under ONR award N00014-22-1-2790.
C.N. was partially supported by DFG-TRR181 and GRK-2583.

\printbibliography
\end{document}